\documentclass[12pt]{amsart}
\usepackage{euscript}
\usepackage{a4wide}
\usepackage[T2A]{fontenc}
\usepackage[utf8]{inputenc}
\usepackage[english]{babel}
\usepackage{amsfonts}
\usepackage{amssymb, amsthm}
\usepackage{amsmath}
\usepackage{graphicx}
\usepackage{geometry}
\usepackage{tikz}

\usepackage{etoolbox}
\patchcmd{\subsubsection}{\itshape}{\bfseries}{}{}

\usepackage{hyperref}
\hypersetup{
    colorlinks=true,
    linkcolor=blue,
    filecolor=magenta,      
    urlcolor=cyan,
    pdfpagemode=FullScreen,
    }
\usetikzlibrary{shapes.geometric}
\usetikzlibrary {patterns}

\renewcommand{\epsilon}{\varepsilon} 
\renewcommand{\phi}{\varphi}

\def\ge{\geqslant}

\def\le{\leqslant}

\def\ol{\overline}
\def\kratno{\lower.5ex\hbox{$\,\vdots\,$}}

\def\eps{\epsilon}
\def\ls{\lesssim}
\def\gs{\gtrsim}

\def\q#1.{\smallbreak\noindent\hskip15pt{\bf#1.}\enspace\ignorespaces} 
\def\dotline{\smallskip\hbox to \hsize{\dotfill}\medskip}

\def\norm[#1]{\left\| #1 \right\|}

\newcommand{\R}{\mathbb{R}}

\newcommand{\K}{\mathcal{K}}
\newcommand{\cI}{\mathcal{I}}
\newcommand{\cJ}{\mathcal{J}}
\newcommand{\bP}{\mathbf{P}}

\newcommand{\cR}{\mathcal{R}}
\newcommand{\bK}{\mathbf{K}}
\newcommand{\bm}{\mathbf{m}}

\newcommand{\T}{\mathbb{T}}
\newcommand{\D}{\mathbb{D}}
\newcommand{\Cm}{\mathbb{C}}

\newcommand{\ilim}{\int\limits}

\renewcommand{\Im}{\mathop{\mathrm{Im}}\nolimits}
\renewcommand{\Re}{\mathop{\mathrm{Re}}\nolimits}
\newcommand{\Res}{\mathop{\mathrm{Res}}\nolimits}
\newcommand{\dm}{\mathop{d\mathrm{m}}\nolimits}
\newcommand{\mm}{\mathop{\mathrm{m}}\nolimits}

\newcommand{\loc}{\mathop{\mathrm{loc}}\nolimits}

\renewcommand{\d}{\partial}
\newcommand{\dist}{\mathop{\mathrm{dist}}\nolimits}
\newcommand{\trace}{\mathop{\mathrm{trace}}\nolimits}

\newcommand{\Szego}{Szeg\H{o} }

\theoremstyle{plain}
\newtheorem{thm}{Theorem} %[section], чтобы нумеровать сначала в каждом разделе
\newtheorem{knownthm}{Theorem} %[section], чтобы нумеровать сначала в каждом разделе

\newtheorem{lemma}{Lemma}
\theoremstyle{definition}

\newtheorem{cor}{Corollary}

\theoremstyle{remark}
\newtheorem*{rem}{Remark}

\title{Dirac operators with exponentially decaying entropy}
\author{Pavel Gubkin}

\address{
\begin{flushleft}
    Pavel Gubkin: gubkinpavel@pdmi.ras.ru\\\vspace{0.1cm}
    St. Petersburg State University \\
    Universitetskaya nab. 7-9, St. Petersburg, 199034, Russia\\\vspace{0.1cm}
    St. Petersburg Department of Steklov Mathematical Institute\\
    Russian Academy of Sciences\\
    Fontanka 27, St. Petersburg, 191023, Russia
\end{flushleft}
}

\thanks{The work is supported by Ministry of Science and Higher Education of the Russian Federation, agreement 075–15–2022–287}



\begin{document}

\newgeometry{left=20mm,right=20mm, top = 20mm}%\setcounter{tocdepth}{2}
\renewcommand{\theknownthm}{\Alph{knownthm}}
\renewcommand{\theconj}{\Roman{conj}}
\renewcommand{\theknownlemma}{\Alph{knownlemma}}

\begin{abstract}
    We prove that the Weyl function of the one-dimensional Dirac operator on the half-line $\R_+$ with exponentially decaying entropy extends meromorphically into the horizontal strip $\{0\ge \Im z > -\delta\}$ for some $\delta > 0$ depending on the rate of decay. If the entropy decreases very rapidly then the corresponding Weyl function turns out to be meromorphic in the whole complex plane. In this situation we show that poles of the Weyl function (scattering resonances) uniquely determine the operator. 
\end{abstract}

\maketitle
%\tableofcontents

\section{Introduction}
The main object of the present paper is the one-dimensional Dirac operator $D_Q$ on the positive half-line $\R_+ = [0,\infty)$,
\begin{gather}\label{one dimensional Dirac operator of the half-line}
    D_Q = J\frac{d}{dr} + Q,
\end{gather}
where $ Q = \left(\begin{smallmatrix}-q&p\\p&q\end{smallmatrix}\right)$ is a symmetric $2\times 2$ zero-trace real potential and $J = \left(\begin{smallmatrix}0&-1\\1&0\end{smallmatrix}\right)$ is the square root of the minus identity matrix.  We will use the entropy approach developed by R. Bessonov and S. Denisov in \cite{Bessonov2020}, \cite{Bessonov2021}, \cite{Bessonov2021entropy} to prove a version of the theorem by P. Nevai and V. Totik from \cite{Nevai1989} for Dirac operator \eqref{one dimensional Dirac operator of the half-line}. We start with the formulation of the original  Nevai-Totik theorem to explain the motivation of the work and then introduce all necessary objects from the spectral theory of Dirac operators and state the results of the present paper. 
\bigskip

\subsection{Orthogonal polynomials on the unit circle. Nevai-Totik theorem}\label{Section orthogonal polynomials}
The basics of the theory of orthogonal polynomials on the unit circle can be found in the books \cite{Szego} by G. \Szego and \cite{Simon} by B. Simon. In this section we give some definitions and formulate the result by P. Nevai and V. Totik from \cite{Nevai1989}, where they proved that the rate of exponential decay of the recurrence coefficients is the same as the radius of analyticity of the inverse \Szego function.
\medskip

Let $\mu =w\dm + \mu_s$ be a probability measure on the unit circle $\T = \{\zeta\colon |\zeta| = 1\}$; here $\mm$ is a normalized Lebesgue measure on the unit circle,  $w$ and $\mu_s$ are the density and the singular part of $\mu$ with respect to $\mm$. Assume that $\mu$ is nontrivial, i.e., that the support of $\mu$ is not a finite set. Then the functions $1,z,z^2\ldots$ are linearly independent in $L^2(\mu)$ and therefore there exists a sequence $\{\Phi_n\}_{n\ge 0}$ of monic polynomials orthogonal in $L^2(\mu)$. They satisfy
\begin{gather*}
    \Phi_{n + 1}(z) = z\Phi_n(z) - \ol{\alpha_n}  \Phi_n^*(z),\quad z \in \Cm,
\end{gather*}
where $\Phi_n^*(z) = z^n\ol{\Phi_n(1/\ol{z})}$ and $\alpha_n\in \D=\{\omega\colon |\omega| < 1\}$ -- the recurrence coefficients. The \Szego theorem states that $\mu$ belongs to the \Szego class, i.e., $\displaystyle\log w \in L^1(\T)$ if and only if $\sum_{n\ge 0}|\alpha_n|^2 < \infty$. Moreover, in this case for every $z\in\D$ the sequence $\Phi_n^*(z)/\norm[\Phi_n^*]_{L^2(\mu)}$ converges to $\Pi(z)$, where $\Pi$ is an outer function in the unit disc satisfying ${|\Pi(\zeta)|^{-2}=w(\zeta)}$ for almost every $\zeta\in \T$, see Theorem 2.3.5 in \cite{Simon}. If $\mu$ is an a.\,c.\ measure from the \Szego class on the unit circle let
\begin{gather*}
    r_{\Pi} = \inf\left\{r \colon \Pi \text{ is analytic in } \{|w| < r^{-1}\}\right\}
\end{gather*}
be the inverse radius of analyticity of $\Pi$. Otherwise put $r_{\Pi} = 1$. Also let
\begin{gather*}
    r_{\alpha} = \limsup_{n\to\infty} |\alpha_n|^{1/n}
\end{gather*}
be the rate of exponential decay of the recurrence coefficients.
\begin{knownthm}[Nevai-Totik theorem; Theorem 1, \cite{Nevai1989}; Chapter 7, \cite{Simon}]\label{theorem Nevai Totik}
        For any nontrivial probability measure $\mu$ on the unit circle we have $r_{\alpha} = r_{\Pi}$.
\end{knownthm}
Methods of the theory of orthogonal polynomials on the unit circle can be applied to the spectral theory of Dirac operators by means of Krein systems, see Section \ref{section Krein systems}. In the original proof of Theorem \ref{theorem Nevai Totik} the relation $\alpha_n = -\ol{\Phi_{n + 1}(0)}$ plays a significant role. It does not have an analogue for Krein systems. Let us give a prove of the inequality $r_{\alpha} \le r_{\Pi}$ in a way that will later work in the setting of spectral theory. The idea of the proof is similar to the idea in Chapter 12.3  of the \Szego's book \cite{Szego} where the asymptotic behaviour of orthogonal polynomials is studied via the Christoffel minimizing functions
\begin{gather}
    \label{Christoffel minimizing function OPUC}
    \lambda_n (z) = \lambda_n (\mu, z) = \min\left\{\norm[P]^2_{L^2(\mu)}\colon \deg P\le n,\; P(z) = 1 \right\},\quad z\in \Cm.
\end{gather}
If $r_{\Pi} = 1$ then the claim $r_{\alpha} \le r_{\Pi}$ follows immediately.
Assume that $r_{\Pi} < 1$, then $\mu$ belongs to the \Szego class, $\Pi$ is well-defined and $d\mu(\zeta) = |\Pi(\zeta)|^{-2}\dm$.
We have, see Chapter 2.2 in \cite{Simon},
\begin{gather}
    \label{lambda n formula}
    \lambda_n(0) = \prod\limits_{n = 0}^{n-1}(1 - |\alpha_n|^2),
    \qquad
    \lambda_{\infty}(0) = \inf_n\lambda_n(0)  = |\Pi(0)|^{-2}.
\end{gather}
Fix an arbitrary number $r_0\in [1, r_{\Pi}^{-1})$. The Taylor series of $\Pi$ converges absolutely in $\{|z| \le r_0\}$ and $\Pi$ does not have zeroes on $\T$, because the converse contradicts the assumption $\mu(\T) < \infty$. Hence for $\zeta\in \T$ we  the uniform bound
\begin{gather}
    \label{taylor polynomial convergence}
    \left|\frac{\Pi(\zeta) - h_n(\zeta)}{\Pi(\zeta)}\right| =  O\left(r_0^{-n}\right),\quad n\to\infty
\end{gather}
holds, where $h_n$ is the $n$-th Taylor polynomial of $\Pi$. Substituting $h_n/h_n(0)$ into \eqref{Christoffel minimizing function OPUC}, we get
\begin{align*}
    \lambda_n(0) &\le \norm[\frac{h_n}{h_n(0)}]^2_{L^2(\mu)}
    = \frac{1}{|h_n(0)|^2}\int_{\T} \left|h_n(\zeta)\right|^2d\mu(\zeta)
    \\
    &= \frac{1}{|\Pi(0)|^2}\int_{\T} \left|h_n(\zeta)\Pi^{-1}(\zeta)\right|^2\dm(\zeta)  = \frac{1}{|\Pi(0)|^2}\int_{\T} \left|1 + \frac{h_n(\zeta) - \Pi(\zeta)}{\Pi(\zeta)}\right|^2\dm(\zeta) 
    \\
    &=\frac{1}{|\Pi(0)|^2}\int_{\T} 1 + 2\Re\left(\frac{h_n(\zeta) - \Pi(\zeta)}{\Pi(\zeta)}\right) + \left|\frac{h_n(\zeta) - \Pi(\zeta)}{\Pi(\zeta)}\right|^2\dm(\zeta).
\end{align*}
Since $\displaystyle \frac{h_n - \Pi}{\Pi}$ is analytic in $\D$, the second term vanishes after the integration. Therefore we have
\begin{align*}
    \frac{1}{|\Pi(0)|^2}\stackrel{\eqref{lambda n formula}}{=}\lambda_n(0)&\le \frac{1}{|\Pi(0)|^2}\int_{\T} 1 +  \left|\frac{h_n(\zeta) - \Pi(\zeta)}{\Pi(\zeta)}\right|^2\dm(\zeta) \stackrel{\eqref{taylor polynomial convergence}}{=} \frac{1}{|\Pi(0)|^2} + O\left(r_0^{2n}\right),\quad n\to\infty. 
\end{align*}
We see that $\lambda_n(0) = \prod_{n = 0}^{n-1}(1 - |\alpha_n|^2)$ is exponentially close to $\lambda_{\infty}(0) = \prod_{n = 0}^{\infty}(1 - |\alpha_n|^2)$. Hence
%The latter together with \eqref{lambda n formula} implies 
$|\alpha_n| =O\left(r_0^{-n}\right) $ as $n\to\infty$. Consequently
$r_{\alpha}\le r_0^{-1}$ and the inequality $r_{\alpha} \le r_{\Pi}$ follows because $r_0$ can be taken arbitrary close to $r_{\Pi}^{-1}$.
\medskip

\medskip
The sequence $\lambda_n(0)$ is connected to the entropy function $\K(n) = -\log \prod_{k\ge n}(1 - |\alpha_n|^2)$ introduced by R. Bessonov and S. Denisov in \cite{Bessonov2021entropy}. They also transferred $\K(n)$ to the theory of Krein systems; this object will be the key tool of the present paper.

\medskip

D. Damanik and B. Simon in the paper \cite{Damanik2006} proved a version of the Nevai-Totik theorem for orthogonal polynomials on the real line. They showed the equality between the degree of exponential decay of Jacobi parameters $a_n, b_n$, i.e., $\limsup\left(|a_n - 1| + b_n\right)^{1/n}$, and the inverse radius of analyticity of the corresponding Jost solution $u$, for details see \cite{Damanik2006} or Section 13.7 in \cite{SimonPart2}.
\medskip

\subsection{Main results}
Recall that we are studying one-dimensional Dirac operator \eqref{one dimensional Dirac operator of the half-line}. Throughout this paper we assume that the entries $p,q$ of $Q$ are real-valued and $p,q\in L^1_{\loc}(\R_+)$ (sometimes we will simply write $Q\in L^1_{\loc}(\R_+)$). The latter means that $p, q\in L^1([0,r])$ for every $ r > 0$. Let us introduce the key objects of the spectral theory of Dirac operators; for a reference we use the book \cite{LevitanSargsjan}  by B. Levitan and I. Sargsjan. Consider the corresponding boundary value problem:
\begin{gather}\label{Dirac system}
    JN'(t,\lambda) + Q(t)N(t,\lambda) = \lambda N(t,\lambda),\quad N(0,\lambda) =  \left(\begin{smallmatrix}1&0\\0&1\end{smallmatrix}\right),\quad t\ge 0.
\end{gather}
Let $\theta_{\pm}$ and $\phi_\pm$ be the entries of $N$, $N(t,\lambda) = \left(\begin{smallmatrix} \theta_+(t,\lambda)&\phi_+(t,\lambda)\\\theta_-(t,\lambda)&\phi_-(t,\lambda)\end{smallmatrix}\right)$.
For any potential $Q$ there exists  a unique Borel measure $\sigma_Q$ on the real line  such that $(1 + x^2)^{-1}\in L^1(\R,d\sigma_Q) $
and the mapping
\begin{gather}\label{Dirac spectral measure isometry}
    (f_1, f_2) \mapsto \frac{1}{\sqrt{\pi}}\int_{\R_+} f_1(x)\theta_+(x,\lambda) + f_2(x)\theta_-(x,\lambda)\, dx,\quad f_1,f_2\in L^2(\R_+)
\end{gather}
is an isometric operator between $L^2(\R_+)\oplus L^2(\R_+) $ and $L^2(\R, \sigma_Q)$.
The measure $\sigma_Q$ is called the spectral measure of $D_Q$. The Weyl function of $D_Q$ is an analytic function in the upper half-plane $\Cm_+=\{z\colon \Im z > 0\}$ defined by the relation
\begin{gather}\label{Dirac Weyl function definition}
    m(z) = -\lim_{t\to\infty} \frac{\phi_+(t,z)}{\theta_+(t,z)}, \quad z\in \Cm_+.
\end{gather}
The spectral measure and the Weyl function are connected by the relation
\begin{gather}\label{spectral measure and weyl function}
    m(z) - m(z_0) =\frac{1}{\pi} \int_{\R}\left(\frac{1}{\lambda - z} - \frac{1}{\lambda - z_0}\right)\,d\sigma_Q(\lambda),\quad z,z_0\in\Cm_+.
\end{gather}
\medskip

A Borel measure $\sigma = w\,dx + \sigma_s$ on the real line $\R$ belongs to the \Szego class if
\begin{gather}\label{Szego class definition}
    \int_{\R}\frac{d\sigma(x)}{1 + x^2} <\infty,\qquad \int_{\R}\frac{\log w(x)}{1 + x^2}dx > -\infty.
\end{gather}
Because of the inequality $\log w < w$, the second integral can't diverge to $+\infty$ and therefore for any measure in the \Szego class we have $\frac{\log w}{1 + x^2}\in L^1(\R)$. Furthermore, in this case there exists an outer function $\Pi$ in $\Cm_+$ such that $\Pi(i) > 0$ and  $|\Pi(x)|^{-2} =w(x)$
for Lebesgue almost every point $x$ on the real line (see Section 4 in \cite{Garnett}). 
We will call $\Pi$ the inverse \Szego function of $\sigma$.
\medskip

Denote the solution of \eqref{Dirac system} for $\lambda = 0$ by $N_Q(t) = N_Q(t,0)$. Let also 
\begin{gather}
    \label{hamiltonian for dirac}
    H_Q(t) = N_Q^*(t)N_Q(t),\quad E_Q(r) = \det\int_{r}^{r + 2}  H_Q(t)\,dt - 4.
\end{gather}
As we will see in Section \ref{section Entropy of Canonical systems}, the matrix $H_Q$ is the  Hamiltonian of the canonical system corresponding to $D_Q$ and $E_Q$ is the leading term of the entropy function $\bK_H$ of this Hamiltonian. If $Q\in L^2(\R_+)$ then $E_Q$ can be bounded via the Sobolev norm of the function $p + iq$, for details see Theorem \ref{theorem Sobolev norm} below.
For $Q\in L^1_{\loc}(\R_+)$ we prove that the exponential decay of $E_Q$ implies the existence of analytic continuation of $m$ into a strip in the lower half-plane; we denote the half-plane $\{z\colon \Im z > -\Delta\}$ by $\Omega_{\Delta}$.  
\begin{thm}\label{Entropy for extension}
Let $p,q\in L^1_{\loc}(\R_+)$ be real-valued functions and $Q = \left(\begin{smallmatrix}-q&p\\p&q\end{smallmatrix}\right)$. Assume that there exists $\delta > 0$ such that $E_Q(r) = O\left(e^{-\delta r}\right)$ as $r\to\infty$. Then 
\begin{enumerate}
    \item \label{entropy for extension part 1} the spectral measure of $D_Q$ is absolutely continuous and  belongs to the \Szego class \eqref{Szego class definition};
    \item\label{entropy for extension part 2}   the inverse \Szego function of $D_Q$ continues analytically into $\Omega_{\delta/4}$;
    \item\label{entropy for extension part 3}  the Weyl function of $D_Q$ continues meromorphically into $\Omega_{\delta/4}$.
\end{enumerate}
In particular, if $E_Q(r) = O\left(e^{-\delta r}\right)$ for every $\delta > 0$, then the \Szego function and the Weyl function of $D_Q$ extend into the whole complex plane $\Cm$.
\end{thm}
We prove Theorem \ref{Entropy for extension} and give details on the continuations of $\Pi$ and $m$ in Section \ref{section proof of Theorem 1}. The constant $1/4$ in the theorem is sharp for some potentials, see Section \ref{section: on the optimality of the constants}.
\medskip

The simplest class of potentials with exponentially decaying entropy is a class of compactly supported potentials. Furthermore, the entropy function of exponentially decreasing or super-exponentially decreasing potentials, i.e., potentials that satisfy
\begin{gather*}
    |p(r)| + |q(r)| = O\left(e^{-ar}\right),\quad r\to\infty,
\end{gather*}
for some $a > 0$ or for all $a > 0$ is also exponentially decaying, see Theorem \ref{theorem on sufficient conditions} below. Theorem \ref{Entropy for extension} is widely known in these situations, see \cite{Klein1986}, \cite{Sjostrand1990}, \cite{Simon2000}, \cite{Froese1997}. When $m_Q$ extends meromorphicaly below the real line one can speak about the scattering resonances. Namely, $z$ is a resonance of $D_Q$ of multiplicity $n$ if $m$ has a pole of order $n$ in $z$,  see the book \cite{Dyatlov2019} by S. Dyatlov and M. Zworski for the general background on resonances and the papers \cite{Iantchenko2014}, \cite{Iantchenko20141}, \cite{Korotyaev2021} by E. Korotyaev for the case of Dirac operators. If $E_Q$ decays super-exponentially then, by Theorem \ref{Entropy for extension}, $m_Q$ is meromorphic in the whole complex plane and the resonances are also well-defined. Let us introduce the class
\begin{gather}\label{definition of class E}
    \mathcal{E} = \bigcup_{\alpha > 1}\mathcal{E}_{\alpha} ,\;\;\;\mathcal{E}_{\alpha} = \left\{Q\in L^1_{\loc}(\R_+)\colon E_Q = O\left( \exp(-{r^{\alpha}})\right),\;r\to\infty\right\}.
\end{gather}
 In Section \ref{section: resonances} we show that the resonances are exactly the zeroes of the corresponding \Szego function $\Pi$ if $E_Q$ decreases super-exponentially and prove the following result. 
\begin{thm}\label{theorem on resonances}
Resonances of the Dirac operator uniquely determine its potential in the class $\mathcal{E}$.   
\end{thm}
Description of the sets which can be resonances sets is an open problem. We plan to continue working in this direction.
\subsection{Square summable potentials}
Denote by $\Delta_{E}$ the rate of exponential decay of the entropy function, i.e., let
\begin{gather*}
    \Delta_{E} = \sup\left\{\delta\colon E_Q(r) = O\left(e^{-\delta r}\right),\quad r\to\infty\right\}.
\end{gather*}
Furthermore, recall that in our notation $\Omega_{\delta}=\{z\colon\Im z > -\delta\} $ and define 
\begin{gather*}
    \Delta_{\Pi} = \sup\left\{\delta\colon \Pi \mbox{ extends analyticaly into } \Omega_{\delta}\mbox{ such that }\frac{\Pi(x -i\delta )}{x + i}\in H^2(\Cm_+)\right\}
\end{gather*}
if $\sigma$ is an absolutely continuous measure from the \Szego class \eqref{Szego class definition} and let $\Delta_{\Pi} = 0$ otherwise. The numbers $\Delta_E$ and $\Delta_{\Pi}$ play the roles of $r_{\alpha}$ and $r_{\Pi}$ from Nevai-Totik theorem \ref{theorem Nevai Totik}. Recall that $r_{\alpha} = r_{\Pi}$; in the following theorem we state that $\Delta_E$ and $\Delta_{\Pi}$ satisfy the two-sided inequality. In this sense Theorem \ref{theorem on equivalence} can be regarded as a version of Theorem \ref{theorem Nevai Totik} for Dirac operators. 
\begin{thm}\label{theorem on equivalence}
Let $p, q\in L^2(\R_+)$ be real-valued functions.
Then $ 4 \Delta_{\Pi}\le \Delta_{E}\le 8\Delta_{\Pi}$.
\end{thm}

\medskip

If $Q\in L^2(\R_+)$, then $\Delta_E$ can be computed. Let $\mathcal{F}$ be the isometric Fourier transform on the real line and define Sobolev space of tempered distributions as
\begin{gather*}
    W^{-1}_2(\R) = \left\{f\colon \norm[f]^2_{W^{-1}_2} = \int_{\R}\frac{(\mathcal{F}f)^2(\eta)}{1 + \eta^2}d\eta 
    < \infty \right\}.
\end{gather*}
\medskip

Theorem 4.1 in \cite{Bessonov20221} states that the Sobolev norm of the function $p + iq$ is comparable to $\bK_H(r) = \sum_{n\ge 0} E_Q(r + n)$. More precisely, denote by $\mathbf{1}_{[r,\infty)} $ the indicator function of the set $[r,\infty)$ and let $f_r = (p + iq)\mathbf{1}_{[r,\infty)}$. Then the inequality $C^{-1}\norm[f_r]_{W^{-1}_2}^2\le \bK_H(r)\le C \norm[f_r]_{W^{-1}_2}^2$
holds with some constant $C$ depending only on $\norm[p + iq]_{L^2(\R_+)}$. It follows that 
\begin{align*}
    \Delta_{E} = 
    \sup\left\{\delta\colon \norm[f_r]_{W^{-1}_2}^2 = O\left(e^{-\delta r}\right),\quad r\to\infty\right\}= 2\liminf_{r\to\infty}\left(-\frac{1}{r}\ln\norm[f_r]_{W^{-1}_2}\right) .
\end{align*}
Therefore Theorem \ref{theorem on equivalence} can be rewritten in the following form.
\begin{thm}\label{theorem Sobolev norm}
Let $p, q\in L^2(\R_+)$ be real-valued functions and let $f_r = (p + iq)\mathbf{1}_{[r,\infty)}$.
Then
\begin{gather*}
    2 \Delta_{\Pi}\le \liminf_{r\to\infty}\left(-\frac{1}{r}\ln\norm[f_r]_{W^{-1}_2}\right) \le 4\Delta_{\Pi}.
\end{gather*}
\end{thm}
\medskip

Proof of Theorem \ref{theorem on equivalence} can be found in Section \ref{proof of Theorem 2}. In Section \ref{section: on the optimality of the constants} we show that the first inequality in Theorem \ref{theorem on equivalence} is sharp for some potentials just as the constant $1/4$ in Theorem \ref{Entropy for extension}. 
\medskip

\subsection{Examples}\label{section examples}
For a given $Q\in L^1_{\loc(\R_+)} $(especially, for $Q\notin L^2(\R_+)$) it can be hard to answer whether or not $E_Q$ decreases fast enough. The following theorem provides some examples of potentials with a rapidly decaying entropy.
\begin{thm}\label{theorem on sufficient conditions} In the two following situations, the assertion $E_Q(r) = O\left(e^{-\delta r}\right)$  as $r\to\infty$ holds.
\begin{itemize}
    \item  $Q =\left(\begin{smallmatrix}0&p\\p&0\end{smallmatrix}\right)$, where $p$ is real-valued and
    \begin{gather*}
        \sup_{t\ge r}\left|\int_{r}^{t}p(s)\,ds\right| = O\left(e^{-\delta r/2}\right),\quad r\to\infty;
    \end{gather*}
    \item $Q=\left(\begin{smallmatrix}-q&p\\p&q\end{smallmatrix}\right)$, where $p$ and $q$ are real-valued, $\displaystyle{\sup_{r\ge 0}\int_r^{r + 1}\left(|p| + |q|\right)\, ds < \infty}$ and 
    \begin{gather*}
        \sup_{t\ge r}\left|\int_{r}^{t}p(s)\,ds\right| = O\left(e^{-\delta r}\right),\quad \sup_{t\ge r}\left|\int_{r}^{t}q(s)\,ds\right| = O\left(e^{-\delta r}\right),\quad r\to\infty.
    \end{gather*}
\end{itemize}
\end{thm}
The fact that the oscillation may compensate the growth of the potential and lead to the properties typical to the properties of decreasing potentials is not new, see \cite{Matveev1972}, \cite{Skriganov1973}, \cite{Sasaki2007} and  Appendix $2$ to XI.$8$ in \cite{ReedSimon3}.
\medskip
If we consider the function $p(x) = e^x\sin\left(e^{2x}\right)$ then
\begin{gather*}
    \sup_{t\ge r}\left|\int_r^t p(x)\, dx \right|= \sup_{t\ge r}\left|\int_r^t e^x\sin\left(e^{2x}\right)\, dx \right| = \sup_{t\ge r}\left|\int_{e^{2r}}^{e^{2t}} \frac{\sin\left(s\right)}{2\sqrt{s}}\, ds \right| =O\left(e^{-r}\right),\quad r\to\infty.
\end{gather*}
Hence, by Theorems \ref{Entropy for extension} and \ref{theorem on sufficient conditions} the Weyl function of $D_Q$ with $Q =\left(\begin{smallmatrix}0&p\\p&0\end{smallmatrix}\right)$ is meromorphic in the half-plane $\{z\colon \Im z > -1/2\}$. A similar argument for $p(x) = xe^{x^2}\sin\left(e^{2x^2}\right)$ gives an example of a ``large'' potential corresponding to a meromorphic Weyl function in the whole complex plane.

\subsection{Structure of the paper}
In Section \ref{section: preliminary information} we introduce the main tools of the present paper -- Krein systems, canonical Hamiltonian systems and the regularized Krein systems, see Sections \ref{section Krein systems}, \ref{Section canonical systems} and \ref{Section regularized Krein system} respectively. Proof of Theorem \ref{Entropy for extension} can be found in Section \ref{short proof of the first theorem}; in Section \ref{section Different construction of the analytical extension} we give more details on the extensions from Theorem \ref{Entropy for extension}.  We prove Theorem \ref{theorem on equivalence} in Section \ref{proof of Theorem 2}, the large part of the proof is devoted to a rescaling argument, see Section \ref{section rescaling}. At the end of the paper, in Section \ref{sufficient conditions}, we prove Theorems \ref{theorem on resonances} and \ref{theorem on sufficient conditions} and discuss the sharpness of the constants in Theorems \ref{Entropy for extension} and \ref{theorem on equivalence}.
\subsection{Acknowledgements}
I am grateful to Roman Bessonov for numerous discussions and constant attention to this work.

\section{Preliminaries}\label{section: preliminary information}
\subsection{Krein systems}\label{section Krein systems}
\subsubsection{General definitions}
Let $a\in L^1_{\loc}(\R_+)$ be a complex-valued function on the positive half-line $\R_+$. The Krein system with the coefficient $a$ is the following system of differential equations:
\begin{align}\label{Krein system}
\begin{cases}
   \frac{\partial}{\partial r}P(r,\lambda) = i\lambda P(r,\lambda) - \ol{a(r)}P_*(r,\lambda),&\quad \,\,P(0,\lambda) = 1,\\
   \frac{\partial}{\partial r}P_*(r,\lambda) = - a(r) P(r, \lambda),&\quad P_*(0,\lambda) = 1.
\end{cases}
\end{align}
After seminal work \cite{Krein} of M. Krein, the solutions of Krein system \eqref{Krein system} are called the continuous analogs of polynomials orthogonal on the unit circle. Using Krein systems, one can transfer methods from the theory of orthogonal polynomials on the unit circle to the spectral problems for self-adjoint differential operators.
Detailed account of this approach can be found in the paper \cite{Denisov2006} by S. Denisov.
\medskip

For any Krein system \eqref{Krein system} there exists a unique Borel measure $\sigma_a$ on the real line such that
$ (1 + x^2)^{-1}\in L^1(\R, \sigma_a)$
and the mapping
\begin{align}\label{spectral measure isometry for Krein system}
    f\mapsto \frac{1}{\sqrt{2\pi}}\int_0^{\infty}f(r)P(r,\lambda)\,dr,
\end{align}
initially defined on simple measurable functions with compact support, can be continuously extended to an isometry from $L^2(\R_+)$ to $L^2(\R, \sigma_a)$. This measure is called the spectral measure of the Krein system. The following theorem is called the Krein theorem and can be regarded as an analog of the \Szego theorem in the setting of the Krein systems.
\begin{knownthm}[Krein theorem; Section 8 in \cite{Denisov2006}; \cite{Teplyaev2005}]\label{Krein theorem}
    Let $\sigma_a$ be the spectral measure of Krein system \eqref{Krein system} and let $P, P_*$ be its solutions. Then the following assertions are equivalent:
    \begin{enumerate}
        \item[$(a)$] $\sigma_a$ belongs to Szeg\H{o} class \eqref{Szego class definition} on the real line;
        \item[$(b)$] \label{part b of Krein theorem } for some point $\lambda_0$ in $\Cm_+ = \{z\in\Cm\colon \Im z > 0\}$ we have
        \begin{align*}
            \int_0^{\infty}|P(r,\lambda_0)|^2 \, dr <\infty;
        \end{align*}
        \item[$(c)$] \label{part C of Krein theorem }
        there exist a sequence $r_n\to\infty$ and a number $\gamma\in [0,2\pi)$ such that for every $\lambda\in \Cm_+$ the limit
        \begin{gather*}
            \Pi(\lambda) = e^{-i\gamma}\lim_{n\to\infty}P_*(r_n,\lambda)
        \end{gather*}
        exists and defines an analytic in $\Cm_+$ function with $\Pi(i) > 0$.
    \end{enumerate}
\end{knownthm}
If the equivalent assertions of the Krein theorem hold then (see Lemma 8.6 in \cite{Denisov2006}) $\Pi$ is the inverse \Szego function of $\sigma_a$, i.e., $\Pi$ is an outer function in $\Cm_+$ satisfying $|\Pi(x)|^{-2} = \sigma_a'(x)$
almost everywhere on $\R$, and
\begin{align}\label{Pi_definition}
    \Pi(\lambda) = \exp\left[ -\frac{1}{2\pi i}\ilim_{-\infty}^{\infty}\left(\frac{1}{s - \lambda} - \frac{s}{s^2 + 1}\right)\log \sigma_a'(s)\, ds\right],\quad\lambda\in\Cm_+.
\end{align}
We will use the following enhancement of assertion $(c)$ of the Krein theorem.
\begin{lemma}[Section 8, \cite{Denisov2006}]\label{lemma on condition (c) of Krein Theorem}
 Assume that $\sigma_a$ belongs to the \Szego class on the real line and the sequence $r_n\to\infty$ is such that $P(r_n, \lambda_0)\to 0$ as $n\to\infty$ for some point $\lambda_0\in \Cm_+$. Then there exists a constant $\gamma\in [0,2\pi)$ and a subsequence $n_k$ such that $P_*(r_{n_k}, \lambda)\to e^{i\gamma}\Pi(\lambda)$ as $k\to\infty$ for every $\lambda\in \Cm_+$.
\end{lemma}
\begin{proof}
The claim immediately follows from Lemma 8.5 and the proof of Lemma 8.6 from~\cite{Denisov2006}.
\end{proof}

\subsubsection{Reproducing kernels and the minimization problem.}\label{Section reproducing kernels and the minimization problem} 
Let $PW_{[0, r]}$ denote the Paley-Wiener space of entire functions $f$ that can be represented in the form
\begin{align*}
    f(z) = \int_0^r\phi(s)e^{izs}\,ds,\quad z\in\Cm,\quad\phi\in L^2[0,r].
\end{align*}
The function 
\begin{gather}\label{formula for reproducing kernel via Krein system}
    k_r(z',z) = \frac{1}{2\pi}\int_0^r P(s,z)\ol{P(s,z')}\, ds
\end{gather}
is the reproducing kernel in $PW_{[0,r]}$ at the point $z'$, see Lemma 8.1 in \cite{Denisov2006}. In other words, for every $f\in PW_{[0,r]}$ we have
\begin{gather*}
    f(z') = \langle f, k_r(z',\cdot)\rangle_{L^2(\sigma_a)} = \int_{\R} f(x)\ol{k_r(z',x)}\, d\sigma_a(x).
\end{gather*}
For $r >0$, define
\begin{align}\label{definition of the minimization function for Krein system}
    \bm_r(z) = \bm_r(\sigma_a, z) &= \inf\left\{\frac{1}{2\pi}\norm[f]^2_{ L^2(\sigma_a)} \colon  f\in PW_{[0,r]},\; f(z) = 1\right\},\quad z\in\Cm.
\end{align}
The function $\bm_r$ is the analog of the Christoffel function \eqref{Christoffel minimizing function OPUC}. Lemma 8.2 in \cite{Denisov2006} says that
\begin{gather}
\label{m_r function formula}
    \bm_r(z) = \left(2\pi k_r(z,z)\right)^{-1} = \left(\int_0^r|P(s, z)|^2\,ds\right)^{-1},\quad z\in \Cm.
\end{gather}
For every $\lambda,\mu\in\Cm$, the functions $P$, $P_*$ satisfy the following Christoffel-Darboux formula:
\begin{align}
    \label{first CD formula}
    P(r,\lambda)\ol{P(r,\mu)} - P_*(r,\lambda)\ol{P_*(r,\mu)} &= i(\lambda - \ol{\mu})\int_0^r P(s,\lambda)\ol{P(s,\mu)}\,ds,
    \\
    \label{second CD formula}
    |P_*(r,\lambda)|^2 - |P(r,\lambda)|^2 &= 2\Im \lambda \int_0^r |P(s,\lambda)|^2\, ds,
\end{align}
see Lemma 3.6 in \cite{Denisov2006}.
Furthermore, a simple calculation shows that 
\begin{gather}
    \label{Krein system reflection formula}
    P(r, z) = e^{i z r}\ol{P_*(r,\ol{ z})},\qquad P_*(r, z) = e^{i z r}\ol{P(r,\ol{ z})},\qquad r\ge 0,\quad z\in\Cm.
\end{gather}
Together with Theorem \ref{Krein theorem}, relation \eqref{second CD formula} gives
\begin{gather}
    \label{expression for the szego function}
    |\Pi(\lambda)|^2 = 2\Im \lambda \int_0^{\infty} |P(s,\lambda)|^2\, ds,\quad \lambda\in\Cm_+. 
\end{gather}
Define $\bm_{\infty}(z) = \inf_r \bm_{r}(z)$. It follows that $\bm_{\infty}$ can be represented by
\begin{gather}
    \label{m infinity formula}
    \bm_{\infty}(\lambda) = \left(\int_0^{\infty}|P(r, \lambda)|^2\,dr\right)^{-1} = \frac{2\Im \lambda}{|\Pi(\lambda)|^2},\quad \lambda\in \Cm_+.
\end{gather}

\subsubsection{Connections with the Dirac operator.}\label{section Connections with the Dirac operator}
Consider a dual Krein system, i.e., Krein system \eqref{Krein system} with the coefficient $-a$, and denote its solutions by $\hat P, \hat{P}_*$. It can be verified (see Section 4 in \cite{Denisov2006}) that $\hat P$ and $-\hat P_*$ solve the same differential system \eqref{Krein system}  as $P$ and $P_*$ but with the initial value $\left(\begin{smallmatrix}1\\-1\end{smallmatrix}\right)$. This can be rewritten in the form 
\begin{gather}
    \label{krein system in a matrix form}
    X'(r, \lambda) = 
    \begin{pmatrix}
     i\lambda &-\ol{a(r)}
     \\
     -a(r) & 0
    \end{pmatrix}
    X(r,\lambda),
    \quad 
    X(0,\lambda) = 
    \left(\begin{smallmatrix}
    1 & 1
     \\
     1 & -1
    \end{smallmatrix}\right),
\end{gather}
where
\begin{gather*}
    X(r,\lambda) = 
    \begin{pmatrix}
     P(r,\lambda) & \hat P(r,\lambda)
     \\
     P_*(r,\lambda) & -\hat P_*(r,\lambda)
    \end{pmatrix}.
\end{gather*}
Define 
\begin{gather}
    \label{dirac potential from the krein system}
    p_a(r) = -2\Re a(2r),\qquad q_a(r)  = 2\Im a(2r), \qquad Q_a = 
    \begin{pmatrix}
     -q_a(r) & p_a(r)
     \\
     p_a(r) & q_a(r)
    \end{pmatrix}.
\end{gather}
A calculation shows (see Chapter 13 in \cite{Denisov2006}) that
\begin{gather*}
    Y(r,\lambda) = \frac{e^{-i\lambda r}}{2}
    \left(\begin{smallmatrix}
     1 & 1
     \\
     -i & i
    \end{smallmatrix}\right)
    X(2r,\lambda)
\end{gather*}
solves the differential system 
\begin{gather*}
    \left(\begin{smallmatrix}
    0 & 1\\
    -1 & 0
    \end{smallmatrix}\right)
    %\begin{pmatrix}0&1\\-1&0\end{pmatrix}
    Y'(r,\lambda) + Q_aY(r,\lambda) = \lambda Y(t,\lambda),\quad Y(0,\lambda) =
    \left(\begin{smallmatrix}
    1 & 0
      \\
    0 & -i
    \end{smallmatrix}\right).
\end{gather*}
This differential system differs from Dirac system \eqref{Dirac system} only in the definition of the square root of the minus identity matrix. It can be easily seen that $f = (f_1, f_2)^T$ solves $\left(\begin{smallmatrix}0&1\\-1&0\end{smallmatrix}\right) f + Q_af = \lambda f$ if and only if $f^\# = (f_1, -f_2)$ solves  $J f^\# - Q_af^\# = \lambda f^\#$. Therefore, the fundamental solution of Dirac system \eqref{Dirac system} with the potential $-Q_a$ can be expressed in terms of $P, P_*, \hat P, \hat P_*$:
\begin{gather}\label{dirac and krein system expressions}
    \begin{pmatrix}
    \theta_+(t,\lambda) &\phi_+(t,\lambda)
    \\
    \theta_-(t,\lambda) &\phi_-(t,\lambda)
    \end{pmatrix} = \frac{e^{-i\lambda r}}{2}
    \begin{pmatrix}
     P(2r,\lambda) + P_*(2r,\lambda) &  i\hat P(2r,\lambda) - i\hat P_*(2r,\lambda)
     \\
     iP(2r,\lambda) - i P_*(2r,\lambda) & -\hat P(2r,\lambda) - \hat P_*(2r,\lambda)
    \end{pmatrix}.
\end{gather}

Theorem 13.1 in \cite{Denisov2006} provides a relation between the spectral measures of $D_{-Q_a}$ and the Krein system with the coefficient $a$.  
Our normalization in definitions \eqref{Dirac spectral measure isometry} and \eqref{spectral measure isometry for Krein system} of spectral measures differs from the one in \cite{Denisov2006}. Because of that Theorem 13.1 actually implies the coincidence of the measures $\sigma_{-Q_a}$ and $\sigma_a$.

\medskip

Using relation \eqref{dirac potential from the krein system} the Dirac system can be rewritten in the Krein system and the other way around. We will say that $D_{-Q_a}$ and the Krein system with the coefficient $a$ correspond to each other.

\medskip

By \eqref{Dirac Weyl function definition} and \eqref{dirac and krein system expressions}, for every Dirac system we have
\begin{gather*}
    m(z) = -\lim_{r\to\infty} \frac{\phi_+(r,z)}{\theta_+(r,z)} = - \lim_{r\to\infty}\frac{i\hat P(2r,z) - i\hat P_*(2r,z)}{ P(2r,z) + P_*(2r,z)} = \lim_{r\to\infty}\frac{ i\hat P_*(r,z) - i\hat P(r,z)}{P_*(r,z) + P(r,z)},\quad z\in\Cm_+.
\end{gather*}
In the \Szego case, both $P(\cdot, z)$ and $\hat P(\cdot, z)$ are in $L^2(\R_+)$, consequently there exists a sequence $r_n\to\infty$ such that $P(r_n, z)\to 0$ and $\hat P(r_n, z)\to 0$ as $n\to\infty$. Hence
\begin{gather*}
    m(z) = i\lim_{n\to\infty}\frac{\hat P_*(r_n,z)}{ P_*(r_n,z)},\quad z\in\Cm_+.
\end{gather*}
By Lemma \ref{lemma on condition (c) of Krein Theorem}, we can choose a subsequence $n_k$ such that both numerator and denominator converge. Therefore
there exists a constant $\gamma\in[0,2\pi)$ such that
\begin{gather}\label{Weyl function via the Szego function}
    m(z) = e^{i\gamma}\frac{\hat\Pi(z)}{\Pi(z)},\quad z\in \Cm_+,
\end{gather}
where $\Pi$ and $\hat{\Pi}$ are the inverse \Szego functions of $\sigma_{-Q_a}$ and $\sigma_{Q_a}$ respectively.
The latter equation will be crucial in the proof of part \eqref{entropy for extension part 3} of Theorem \ref{Entropy for extension}.

\subsection{Canonical systems}\label{Section canonical systems}
The entropy approach is based on the reduction of Dirac and Krein systems to a more general form of differential equations -- the canonical Hamiltonian system. Below we introduce key objects and definitions of the theory, for details see \cite{Remling2018} and \cite{Romanov2014}.
\medskip

A Hamiltonian is a matrix-valued mapping of the form 
\begin{gather}\label{canonical system general form of the hamiltonian}
    H = \begin{pmatrix}h_1&h\\h&h_2\end{pmatrix},
\end{gather}
where $h, h_1$ and $h_2$ are real-valued functions from $L^1_{\loc}(\R_+)$. Moreover, it is assumed that $H$ satisfies $\trace H(t) > 0$ and $\det H(t) \ge 0$ for every $t\ge 0$.
A Hamiltonian $H$ is called singular if 
$ \int_{\R_+}\trace H(s)\, ds = +\infty$. We will call a Hamiltonian $H$ trivial if it coincides with $\left(\begin{smallmatrix}1&0\\0&0\end{smallmatrix}\right)$ or $\left(\begin{smallmatrix}0&1\\0&1\end{smallmatrix}\right)$ and nontrivial otherwise.
The canonical system with Hamiltonian $H$ is the differential equation
\begin{gather}\label{hamiltonian canonical system}
    J\frac{\partial}{\partial t} M(t, z)  = zH(t) M(t, z),
    \quad
    M(0, z)  = \left(\begin{smallmatrix}1&0\\0&1\end{smallmatrix}\right),\quad z\in \Cm, \quad t\ge 0.
\end{gather}
The solution $M$ of \eqref{hamiltonian canonical system} is often represented as
\begin{gather}\label{solution of a canonical system}
    M(t, z) =(\Theta(t, z), \Phi(t,z)) = 
    \begin{pmatrix}
    \Theta_+(t, z) & \Phi_+(t,z))
    \\
    \Theta_-(t, z) & \Phi_-(t,z))
    \end{pmatrix}.
\end{gather}
The Weyl function of the canonical system is defined by
\begin{gather*}
    m(z)=\lim_{t\to\infty}\frac{w\Phi_+(t, z) + \Phi_-(t, z)}{w\Theta_+(t, z) + \Theta_-(t, z)} ,\quad z\in\Cm_+,\quad w\in\Cm\cup\{\infty\}.
\end{gather*}
If the Hamiltonian is singular, this limit is correctly defined and does not depend on $w$. 
Furthermore, $m$ has a strictly positive imaginary part in $\Cm_+$ and therefore admits the following Herglotz representation:
\begin{gather}\label{Herglotz representation}
    m(z) = \frac{1}{\pi}\int_{\R}\left(\frac{1}{x - z} - \frac{x}{x^2 + 1}\right)d\sigma(x) + az + b, 
\end{gather}
where $a\in\R$ and $b\ge0$ are constants and $\sigma$ is a Borel measure satisfying $ (1 + x^2)^{-1}\in L^1(\R, \sigma)$. The measure $\sigma$ is called the spectral measure of $H$ and of canonical system \eqref{hamiltonian canonical system}.
\subsubsection{Reduction of the Dirac system to the canonical system.}\label{Section reduction of the Dirac operator to the canonical system}
As we mentioned it earlier, \eqref{hamiltonian canonical system} is a more general form of a differential system than \eqref{Dirac system} or \eqref{Krein system}. In this subsection we outline the reduction of Dirac system \eqref{Dirac system} to the Hamiltonian canonical system. We omit the calculations, for a more detailed explanation see \cite{Romanov2014} or Section 2.4 in \cite{Bessonov2020}.
\medskip

Consider Dirac system \eqref{Dirac system} with the potential $Q$ and define  $N_Q(t)  = N(t, 0)$. Then $M(t,z) = N_Q^{-1}(t)N(t,z)$ is the fundamental solution of canonical system \eqref{hamiltonian canonical system} with the Hamiltonian 
\begin{gather*}
    H_Q(t) = N_Q^*(t)N_Q(t).
\end{gather*}
Moreover, the spectral measure of the canonical system with Hamiltonian $H_Q$ coincides with the spectral measure of the Dirac operator $D_Q$  (see  Section 2.4 in \cite{Bessonov2020}).

\subsubsection{Entropy of a canonical system.} \label{section Entropy of Canonical systems}
In the papers \cite{Bessonov2020}, \cite{Bessonov2021} R. Bessonov and S. Denisov described the class of Hamiltonians for which the corresponding spectral measure belongs to the \Szego class on the real line. They introduced the criterion in terms of the entropy function of the Hamiltonian. Let us define it. \medskip

Consider an arbitrary singular nontrivial Hamiltonian $H$ and let $H_r$ be its $r$-shift, i.e., let $H_r$ be such that $H_r(x) = H(r + x)$, $r,x\ge 0$. Denote by $m_r, \sigma_r$ and $ w_r$ the Weyl function, the spectral measure and the density of the spectral measure corresponding tho the canonical system with the Hamiltonian $H_r$. Next, define
\begin{align*}
    \cI_H(r) = \Im m_r(i),\qquad 
    \cR_H(r) = \Re m_r(i) 
    ,\qquad \cJ_H(r) = \frac{1}{\pi}\int_{\R} \frac{\log(w_r(x))}{1 + x^2}\,dx.
\end{align*}
The entropy function $\K_H$ of $\sigma$ is defined as
\begin{gather}
    \label{definition entropy function of the measure}
    \K_H(r) = \log \cI_H(r) - \cJ_H(r), \quad r\ge 0.
\end{gather}
If $\sqrt{\det H}\notin L^1(\R_+)$ then we can define 
\begin{gather}\label{definition of the entropy for the hamiltonian}
    \bK_H(r) = \sum_{n\ge 0} \left(\det\int_{\eta_{n}(r)}^{\eta_{n + 2}(r)}  H(t)\,dt -4\right),\quad \eta_n(r) = \min\left\{x\colon \int_r^x\sqrt{\det H(t)}dt = n\right\}.
\end{gather}
The function $\bK_H$ is called the entropy function of the Hamiltonian $H$. The main result of \cite{Bessonov2021} is the following theorem.
\begin{knownthm}[Theorem 1.2, \cite{Bessonov2021}]\label{known theorem on two entropies}
    The spectral measure of a singular nontrivial Hamiltonian $H$  belongs to Szegő class \eqref{Szego class definition} if and only if $\sqrt{\det H}\notin L^1(\R_+)$ and $\bK_H(0) <\infty$. Moreover, we have
    \begin{gather*}
        c_1 \K_H(r)\le \bK_H(r) \le c_2 \K_H(r) \cdot e^{c_2 \K_H(r)},
    \end{gather*}
    where $c_1$ and $c_2$ are some absolute constants.
\end{knownthm}
In the present paper we are interested in the case when $ \bK_{H}(r)$ decreases exponentially fast in $r$. It is equivalent to the exponential decay of $\K_H(r)$ and, formally, can be written as
\begin{gather}\label{assertion entropy exponential decay}
    \sum_{n\ge 0} \left(\det\int_{\eta_{n}(r)}^{\eta_{n + 2}(r)}  H(t)\,dt -4\right) = O\left(e^{-\delta r}\right), \quad r\to\infty
\end{gather}
for some $\delta > 0$. We have $\eta_{n + 1}(r) = \eta_n(\eta_1(r))$ hence the latter in its turn is equivalent to the exponential decay of the first term of the sum, i.e., 
\begin{gather*}
    \det\int_{\eta_{0}(r)}^{\eta_2(r)}  H(t)\,dt -4 = \bK_H(r) - \bK_H(\eta_1(r)) = O\left(e^{-\delta r}\right), \quad r\to\infty.
\end{gather*}
Notice that if $H= H_Q$ is the Hamiltonian constructed for the Dirac operator $D_Q$ in  Section \ref{Section reduction of the Dirac operator to the canonical system}, then $\det H(t) = 1$ for every $t\ge 0$ and consequently $\eta_n(r) = n + r$ for every $n,r \ge 0$. In this situation assertion \eqref{assertion entropy exponential decay} becomes
\begin{gather*}
    \det\int_{r}^{r + 2}  H_Q(t)\,dt -4 = O\left(e^{-\delta r}\right), \quad r\to\infty,
\end{gather*}
which is exactly the assertion $E_Q(r) = O\left(e^{-\delta r}\right)$ from Theorem \ref{Entropy for extension}.

\subsection{Regularized Krein systems}\label{Section regularized Krein system}
Fix a singular nontrivial Hamiltonian $H$. Assume that the spectral measure $\sigma$ corresponding to $H$ belongs to the \Szego class and let $\K_H, \cI_H$ and $\cR_H$ be as in the previous section.  In order to simplify the exposition we will omit the index $H$ later on.  The regularized Krein system corresponding to $H$ is the following system of differential equations:
\begin{align}
    \label{diff equation for tilde P_*}
    \frac{\d}{\d r}\Tilde{P}_r^*(z) &= (z - i)f_1(r)\Tilde{P}_r(z) + izf_2(r) \Tilde{P}_r^*(z), &\Tilde{P}_0^*(z)= I_H(0)^{-1},
    \\
    \label{diff equation for tilde P}
    \frac{\d}{\d r}\Tilde{P}_r(z)
    &= iz \Tilde{P}_r(z) + (z + i)\ol{f_1(r)}\Tilde{P}_r^*(z) - iz\ol{f_2(r)}\Tilde{P}_r(z),&\Tilde{P}_0(z) = 0,
\end{align}
where $f_1(r)=-\frac{1}{4}e^{2iu(r)}\left(\frac{\cR'(r)}{\cI(r)} + i\frac{\cI'(r)}{\cI(r)}\right)$ with $u(r) = \int_0^r\frac{\cR'(t)}{2\cI(t)}dt$ and $f_2(r) = \frac{\K'(r/2)}{4}$. Notice that by Lemma 3 in \cite{Bessonov2020}, $\K_H, \cI_H$ and $\cR_H$ are locally absolutely continuous and the differential system is well-defined. 
\medskip

First of all, let us show that the regularized Krein system defined in \cite{Bessonov2020} coincides with the one defined by \eqref{diff equation for tilde P_*} and \eqref{diff equation for tilde P}. Denote by $\Tilde{\bP}_r$ and $\Tilde{\bP}_r^*$ the regularized Krein system from \cite{Bessonov2020}. We claim that $\Tilde{\bP}_r = \Tilde{P}_r $ and $\Tilde{\bP}_r^* = \Tilde{P}_r^* $. Indeed, the initial values of $\Tilde{P}_r$ and $\Tilde{P}_r^*$ is chosen so that it coincides with the initial values of $\Tilde{\bP}_r$ and $\Tilde{\bP}_r^*$ hence it is suffices to show that $\Tilde{\bP}_r$ and $\Tilde{\bP}_r^*$ satisfy differential equations \eqref{diff equation for tilde P_*} and \eqref{diff equation for tilde P}. Equation \eqref{diff equation for tilde P_*} follow from Lemma 8 in \cite{Bessonov2020} by the change of variables. To establish \eqref{diff equation for tilde P}, notice that $(37)$ in \cite{Bessonov2020} yields $\Tilde \bP_r(z) = e^{irz} \ol{\Tilde{\bP}_r^*(\ol{z})}$ and
\begin{align}
    \nonumber
    \frac{\d}{\d r}\Tilde{\bP}_r(z) &= \frac{\d}{\d r}\left(e^{irz}\ol{\Tilde{\bP}^*_r(\ol{z})}\right) = iz\cdot e^{irz}\ol{\Tilde{\bP}^*_r(\ol{z})} + e^{irz}\cdot\frac{\d}{\d r}\ol{\Tilde{\bP}^*_r(\ol{z})} 
    \\
    \nonumber
    &= iz \Tilde{\bP}_r(z) +  e^{irz}\ol{\left((\ol z - i)f_1(r)\Tilde{\bP}_r(\ol{z}) + i\ol{z}f_2(r) \Tilde{\bP}_r^*(\ol{z})\right)} 
    \\
    \label{establishment of the second diff equation}
    &= iz \Tilde{\bP}_r(z) + (z + i)\ol{f_1(r)}\Tilde{\bP}_r^*(z) - iz\ol{f_2(r)}\Tilde{\bP}_r(z).
\end{align}
\subsection{Properties of the regularized Krein systems}\label{Section a careful look at the regularized Krein system}
In this subsection we discuss properties of the regularized Krein system. Lemmas \ref{convergence of the regularized system} and \ref{coefficients bound} are respectively Lemmas 9 and 8 from \cite{Bessonov2020}, in Lemma \ref{coefficients bound} we will need more accurate bounds than provided in \cite{Bessonov2020} so we state it with a proof (which is almost identical to the one in \cite{Bessonov2020}). The following lemma is a simple corollary of the differential equations for $\Tilde{P}_r^*$ and $\Tilde{P}_r$.
\begin{lemma}
The functions $\Tilde{P}_r^*$ and $\Tilde{P}_r$ satisfy the reflection formula
\begin{align}
    \label{inversion formula}
    \Tilde P_r^*(z) &= e^{irz} \ol{\Tilde{P}_r(\ol{z})},\quad \Tilde P_r(z) = e^{irz} \ol{\Tilde{P}_r^*(\ol{z})}.
\end{align} 
Moreover, the  following differential equations hold for the absolute values of $\Tilde P_r$ and $\Tilde P_r^*$: 
\begin{align}
\label{diff equation for modulus of P_*}
    \frac{\d}{\d r}|\Tilde{P}_r^*(z)|^2  &= 2\Re\left((z - i)f_1(r)\Tilde{P}_r(z)\ol{\Tilde{P}_r^*(z)}\right) - 2\Im z f_2(r) |\Tilde{P}_r^*(z)|^2,
    \\
    \label{diff equation for modulus of P}
    \frac{\d}{\d r}|\Tilde{P}_r(z)|^2 
    &=-2\Im z |\Tilde{P}_r(z)|^2 + 2\Re\left((z + i)\ol{f_1(r)}\Tilde{P}_r^*(z)\ol{\Tilde{P}_r(z)} \right) + 2\Im z f_2(r) |\Tilde{P}_r(z)|^2.
\end{align} 
\end{lemma}
\begin{proof}
Equations in \eqref{inversion formula} follow from \eqref{establishment of the second diff equation}.  The rest of the proof is a straightforward calculation. Indeed, we have
\begin{align*}
    \frac{\d}{\d r}|\Tilde{P}_r^*(z)|^2 &= 2\Re\left(\ol{\Tilde{P}_r^*(z)}\frac{\d}{\d r}\Tilde{P}_r^*(z)\right) = 2\Re\left(\ol{\Tilde{P}_r^*(z)}\left((z - i)f_1(r)\Tilde{P}_r(z) + izf_2(r) \Tilde{P}_r^*(z)\right)\right)
    \\
    &=2\Re\left((z - i)f_1(r)\Tilde{P}_r(z)\ol{\Tilde{P}_r^*(z)}\right) - 2\Im z\Re f_2(r) |\Tilde{P}_r^*(z)|^2;
    \\
    \frac{\d}{\d r}|\Tilde{P}_r(z)|^2 &= 2\Re\left(\ol{\Tilde{P}_r(z)}\frac{\d}{\d r}\Tilde{P}_r(z)\right) 
    \\
    &= 2\Re\left(\ol{\Tilde{P}_r(z)}\left(iz \Tilde{P}_r(z) + (z + i)\ol{f_1(r)}\Tilde{P}_r^*(z)  -iz\ol{f_2(r)}\Tilde{P}_r(z)\right)\right) 
    \\
    &=-2\Im z |\Tilde{P}_r(z)|^2 + 2\Re\left((z + i)\ol{f_1(r)}\Tilde{P}_r^*(z)\ol{\Tilde{P}_r(z)} \right) + 2\Im z\Re f_2(r) |\Tilde{P}_r(z)|^2.   
\end{align*}
\end{proof}
Further in the paper we will use the symbols $\ls$ and $\gs$ meaning that the corresponding inequality $\le$ or $\ge$ holds with some multiplicative constant depending only on fixed parameters. We will use the symbol $\approx$ when both $\ls$ and $\gs$ hold.
\medskip

The spectral measure $\sigma$ belongs to \Szego class \eqref{Szego class definition} hence the inverse \Szego function $\Pi$ is well defined by \eqref{Pi_definition}. The solutions of the regularized Krein system satisfy the following limit relations. 
\begin{lemma}[Lemma 9, \cite{Bessonov2020}]\label{convergence of the regularized system}
For $z\in\Cm_+$, we have
\begin{gather*}
    \lim_{r\to\infty} \Tilde{P}_r^*(z) = \Pi(z),
    \qquad
    \lim_{r\to\infty} \Tilde{P}_r(z) = 0,\qquad
    \int_0^{\infty}|\Tilde{P}_r(z)|^2\, dr < \infty.
\end{gather*}
\end{lemma}
\begin{lemma}[Lemma 8, \cite{Bessonov2020}]\label{coefficients bound}
The coefficient $f_1$ of the regularized Krein system satisfies 
\begin{gather*}
    |f_1(r)| \ls \sqrt{|\K'(r)|} + |\K'(r)|
\end{gather*}
uniformly for every $r\ge 0$.
\end{lemma}
\begin{proof}
Due to the definition of $f_1$, the claim of the lemma is equivalent to 
\begin{gather*}
    \left|\frac{\cI'(r)}{\cI(r)}\right| + \left|\frac{\cR'(r)}{\cI(r)}\right|\ls \sqrt{|\K'(r)|} + |\K'(r)|.
\end{gather*}
Formulas (39) and (40) in \cite{Bessonov2020} give
\begin{align}
    \label{derivative of K' from the paper}
    -\K' &= \left(\cI h_1 + \frac{1}{\cI h_1} - 2\right) + \frac{1}{4}\left(\frac{\cR'}{I}\right)^2\frac{1}{\cI h_1},
    \\
    \label{derivative of I'/I from the paper}
    \frac{\cI'}{\cI} &= \cI h_1 - \frac{1}{\cI h_1} - \frac{1}{4}\left(\frac{\cR'}{\cI}\right)^2\frac{1}{\cI h_1},
\end{align}
where $h_1$ is the upper-left entry of $H$, see \eqref{canonical system general form of the hamiltonian}. The function $\K$ is non-increasing hence $-\K =  |\K'|$ for every $t\ge 0$.  
Two terms in the right hand side of the first equality are nonnegative and therefore we have
\begin{gather}
    \label{rate of convergence for entropy formula 1}
    \cI h_1 + \frac{1}{\cI h_1} - 2\le  |\K'|,
    \\
    \label{rate of convergence for entropy formula 2}
    \frac{1}{4}\left(\frac{\cR'}{\cI}\right)^2\frac{1}{\cI h_1} \le |\K'|,\qquad \left(\frac{\cR'}{\cI}\right)^2\le 4\cI h_1|\K'|.
\end{gather}
From the definition of $\cI(r) $ we know $\cI(r) = \Im m_r(i) \ge 0$ because $m_r$ is a Herglotz function. Also overall assumptions on $H$ imply $h_1\ge 0 $. If $\cI(r)h_1(r)\in\left[\frac{1}{2}, 2\right]$ then by \eqref{rate of convergence for entropy formula 2} we have $\displaystyle \left( {\cR'}/{\cI}\right)^2\ls |\K'|$ and 
\begin{gather}\label{bound for diff Ih case 1}
    \left|\cI h_1 - \frac{1}{\cI h_1}\right|\ls  \sqrt{\cI h_1 + \frac{1}{\cI h_1} - 2}\le \sqrt{|\K'|};
\end{gather}
else we have 
\begin{gather}\label{bound for diff Ih case 2}
    Ih_1 \ls \left|\frac{1}{\cI h_1} -  \cI h_1\right|\approx \cI h_1 + \frac{1}{\cI h_1} - 2 \le |\K'|
\end{gather}
and $\displaystyle \left({\cR'}/{\cI}\right)^2\le 4\cI h_1|\K'| \ls |\K'|^2$. The required inequalities for ${\cR'}/{\cI}$ immediately follow in both situations. To get the bound for $ {\cI'(r)}/{\cI(r)}$ substitute \eqref{bound for diff Ih case 1}, \eqref{bound for diff Ih case 2} and \eqref{rate of convergence for entropy formula 2} into \eqref{derivative of I'/I from the paper}.
\end{proof}
Consider Krein system \eqref{Krein system} and construct a canonical system via the reductions from Sections \ref{section Connections with the Dirac operator} and \ref{Section reduction of the Dirac operator to the canonical system} so that the spectral measure of the Krein system coincides with the spectral measure of the canonical system. The next lemma connects the Krein system and its regularized version.
\begin{lemma}\label{lemma two repr kernels}
For $z_0\in\Cm_+$ and $ r\ge 0$ we have
\begin{gather}\label{formula from lemma with two reproducing kernels}
    2\Im z_0\int_{r}^{\infty} |P(x,z_0)|^2\,dx = |\Pi(z_0)|^2 -  \left(|\Tilde P_{r}^*(z_0)|^2 - |\Tilde P_{r}(z_0)|^2\right).
\end{gather}
\end{lemma}
\begin{proof}
Let $\sigma$ be the spectral measure of the Krein system. Recall that the reproducing kernel in the space $PW_{[0,r]}$ with norm inherited from $L^2(d\sigma)$ at the point $z_0\in \Cm_+$ is given by \eqref{formula for reproducing kernel via Krein system},
\begin{gather*}
    k_r(z_0, z) = \frac{1}{2\pi}\int_0^r P(x, z)\ol{P(x,z_0)}\,dx.
\end{gather*}
On the other hand, the reproducing kernel admits the following representation in terms of the regularized Krein system (see formula (48) in \cite{Bessonov2020}):
\begin{gather*}
    k_r(z_0, z) = -\frac{1}{2\pi i}\frac{\Tilde P_r^*(z)\ol{\Tilde P_r^*(z_0)} - \Tilde P_r(z)\ol{\Tilde P_r(z_0)}}{z  - \ol{z_0}}.
\end{gather*}
Therefore we have
\begin{gather*}
    k_r(z_0, z_0) = \frac{|\Tilde P_r^*(z_0)|^2 - |\Tilde P_r(z_0)|^2}{4\pi\Im z_0} = \frac{1}{2\pi}\int_0^r |P(x, z_0)|^2\,dx.
\end{gather*}
The claim of the lemma now follows from \eqref{expression for the szego function}.
\end{proof}
\section{Analytic extension of the \Szego function. Proof of Theorem \ref{Entropy for extension}}\label{section proof of Theorem 1}
Consider canonical system \eqref{hamiltonian canonical system} with spectral measure $\sigma$ in the \Szego class. Let $\K$ be its entropy function \eqref{definition entropy function of the measure}, $\Pi$ be the inverse \Szego function of $\sigma$ and $\Tilde{P}_{r}, \Tilde{P}_{r}^*$ be the solutions of the corresponding regularized Krein system. Before the proof of Theorem \ref{Entropy for extension} let us show that the singular part of $\sigma$ is absent under the weaker assumptions on $\K$.
\begin{lemma}\label{entropy and absolute continuity}
If $\sqrt{|\K'|}\in L^1(\R_+)$ then $\sigma$ is a.\,c.\,with respect to the Lebesgue measure on the real line and $\Pi$ is continuous in $\ol{\Cm_+}$.  Furthermore, uniformly for $z\in \ol{\Cm_+}$ we have
\begin{align}\label{exponential bound}
    \left|\log|\Tilde{P}_{r}^*(z)|\right| \ls |z| + 1,
    \\
    \label{exponential bound for pi}
    \big|\log|\Pi(z)|\big| \ls |z| + 1.
\end{align}
\end{lemma}
\begin{proof}
Recall differential equation \eqref{diff equation for tilde P_*} and divide it by $\Tilde{P}_r^*(z)$:
\begin{gather*}
    \frac{\d}{\d r}\log\Tilde{P}_r^*(z) = (z - i)f_1(r)\frac{\Tilde{P}_r(z)}{\Tilde{P}_r^*(z)} + izf_2(r).
\end{gather*}
The inequality $|\Tilde{P}_r(z)|\le |\Tilde{P}_r^*(z)|$ holds for $z\in \ol{\Cm_+}$, therefore 
\begin{align*}
    \left|\frac{\d}{\d r}\log\Tilde{P}_r^*(z)\right| &\le \left|(z - i)f_1(r) + izf_2(r)\right| \ls  (|z| + 1)\left(\sqrt{|\K'(r/2)|} + |\K'(r/2)|\right),
\end{align*}
where the last inequality is by Lemma \ref{coefficients bound}. Inequality \eqref{exponential bound} then follows by integration. We see that $\frac{\d}{\d r}\log\Tilde{P}_r^*(z)\in L_1(\R_+)$ and $\norm[\frac{\d}{\d r}\log\Tilde{P}_r^*(z)]_1$ is uniformly bounded on compact subsets of $\ol{\Cm_+}$. This means that $\Tilde{P}_r^*(z)$ converge as $r\to\infty$ uniformly on compact subsets of $\ol{\Cm_+}$. By Lemma \ref{convergence of the regularized system}, the limit coincides with $\Pi$ in $\Cm_+$ therefore $\Pi$ is continuous in $\ol{\Cm_+}$ and \eqref{exponential bound for pi} follows from \eqref{exponential bound}.
Finally, $|\tilde P_r^*(x)|^{-2} dx$ is a spectral measure of the Hamiltonian $\hat{H}_r$ from \cite{Bessonov2020} hence Corollary 5.8 and Theorem 7.3 in \cite{Remling2018} give ${|\Tilde{P}_r^*(x)|^{-2}}dx \stackrel{w^*}{\to} d\sigma$.
 Together with \eqref{exponential bound} this means that $\sigma$ is absolutely continuous. The proof is concluded.
\end{proof}
\subsection{Proof of Theorem \ref{Entropy for extension}}\label{short proof of the first theorem}
\begin{proof}
The Hamiltonian $H = H_Q$ defined by \eqref{hamiltonian for dirac} coincides with the Hamiltonian constructed in Section \ref{Section reduction of the Dirac operator to the canonical system}. At the end of Section \ref{section Entropy of Canonical systems} we showed that the assertion
\begin{gather*}
    E_Q(r) = \det\int_{r}^{r + 2}  H_Q(t)\,dt - 4 = O\left(e^{-\delta r}\right),\quad r\to\infty,
\end{gather*}
of Theorem \ref{Entropy for extension} is equivalent to the assertion
\begin{gather*}
    \bK_{H}(r) = \sum_{n\ge 0}\left(\det\int_{r+ n}^{r+ n + 2}  H_Q(t)\,dt - 4\right) = O\left(e^{-\delta r}\right),\quad r\to\infty.
\end{gather*}
By Theorem \ref{known theorem on two entropies}, it follows that $\sigma$ is in the \Szego class and we have
\begin{gather}
    \label{decay of the entropy function if theorem 1}
    \K_H(r) = O\left(e^{-\delta r}\right),\quad r\to\infty.
\end{gather}
In particular, $\sqrt{|\K'|}\in L_1$. Hence part \eqref{entropy for extension part 1} of Theorem \ref{Entropy for extension} immediately follows from Lemma \ref{entropy and absolute continuity}.
\medskip 

Let us show that $\Tilde{P}_r^*(z)$ converges as $r\to\infty$ uniformly on compact subsets of $\Omega_{\delta/4}$. By Lemma \ref{convergence of the regularized system}, the limit function will be the required continuation of $\Pi$. Let $z$ be with $\Im z \le 0$. Recall differential equation \eqref{diff equation for tilde P_*} and apply the bounds from Lemma \ref{coefficients bound} to it:
\begin{gather*}
    \left|\frac{\d}{\d r}\Tilde{P}_r^*(z)\right| \ls  (|z| + 1)\left(\sqrt{|\K'(r/2)|} + |\K'(r/2)|\right)\left|\Tilde{P}_r(z)\right| + |z||\K'(r/2)|\left| \Tilde{P}_r^*(z)\right|.
\end{gather*}
By \eqref{exponential bound} and reflection formula  \eqref{inversion formula}, we get
\begin{gather*}
    \left|\Tilde P_r^*(z)\right| \le e^{r|\Im z|}e^{c(|z| + 1)},\quad \left|\Tilde P_r(z)\right| \le e^{r|\Im z|}e^{c(|z| + 1)}, \quad \Im z \le 0,
\end{gather*}
where $c$ does not depend on $z$.
Substituting these bounds into the previous inequality, we obtain
\begin{gather}
\label{trivial inequality for the differential equation}
    \left|\frac{\d}{\d r}\Tilde{P}_r^*(z)\right| \ls (|z| + 1)\left(\sqrt{|\K'(r/2)|} + |\K'(r/2)|\right)e^{r|\Im z|}e^{c(|z| + 1)},\quad \Im z \le 0.
\end{gather}
Because of \eqref{decay of the entropy function if theorem 1}, the integral
\begin{gather}
    \label{convergence of entropy with the increasing exponent}
    \int_0^{\infty}\left(\sqrt{|\K'(r/2)|} + |\K'(r/2)|\right)e^{r|\Im z|}\, dr
\end{gather}
converges when $|\Im z| < \delta/4$. Therefore $\frac{\d}{\d r}\Tilde{P}_r^*(z)\in L^1$ for $z$ with  $0 \ge \Im z > -\delta/4$ and Part \eqref{entropy for extension part 2} of Theorem \ref{Entropy for extension} follows.
\medskip

Simple calculations show (see Lemma 3 in \cite{Bessonov2020}) that for every $t\ge 0$  we have
\begin{gather*}
    H_{-Q}(t) = J^*H_Q(t) J,\qquad \bK_{H_{-Q}}(t) = \bK_{H_Q}(t),\qquad E_{-Q}(t) = E_{Q}(t).
\end{gather*}
Hence $-Q$ also satisfies the assertions of the theorem and Part \eqref{entropy for extension part 2} can be applied for $-Q$ as well as for $Q$. It follows that both $\Pi_{Q}$ and $\Pi_{-Q}$ extend analytically into $\Omega_{\delta/4}$ and therefore the Weyl function of $D_Q$ can be meromophically extended into the same domain via relation \eqref{Weyl function via the Szego function}. This concludes the proof of the whole theorem.
\end{proof}
\begin{cor}\label{corollary bound for extended function Pi}
Assume that the assertions on Theorem \ref{Entropy for extension} hold. Then for every $\delta_1 < \delta/4$ there exists a constant $C$ such that 
$|\Pi(z)|\le e^{C(|z| + 1)}$ in $\Omega_{\delta_1}$.
\end{cor}
\begin{proof}
The required inequality for $z\in\Cm_+$ and $z\in \Omega_{\delta_1}\setminus \Cm_+$ follows from \eqref{exponential bound for pi} and \eqref{trivial inequality for the differential equation} respectively.
\end{proof}
Let $P$ be the solution of the Krein system corresponding to $D_Q$. The next corollary is a quantitative version of Lemmas \ref{convergence of the regularized system} and \ref{lemma two repr kernels}.
\begin{cor}\label{corollary which was step 2}
Fix $z_0$ with $\Im z_0 > \delta / 4$. 
Under the assumptions of Theorem \ref{Entropy for extension}, for $r\ge 0$  we have
\begin{gather*}
    |\Tilde{P}_r(z_0)| \ls e^{-\delta r/4},\qquad \left|\Tilde P^*_{r}(z_0) - \Pi(z_0)\right| \ls e^{-\delta r/2},
    \\
    \int_{r}^{\infty} |P(x,z_0)|^2\,dx \ls e^{-\delta r/2}.
\end{gather*}
\end{cor}
\begin{proof}
Integrating \eqref{trivial inequality for the differential equation} for $z = \ol{z_0}$, we get 
\begin{gather*}
    |\Tilde{P}_r^*(\ol{z_0})|\ls 1 +\int_0^{r}\left(\sqrt{|\K'(\rho/2)|} + |\K'(\rho/2)|\right)e^{\rho\Im z_0}\, d\rho.
\end{gather*}
The entropy decay $\K(r) = O\left(e^{-\delta r}\right)$ implies $|\Tilde{P}_r^*(\ol{z_0})| \ls e^{(\Im z_0 - \delta/4)r}$ and the bound for $\Tilde{P}_r(z_0)$ follows from \eqref{inversion formula}. Recall differential equation \eqref{diff equation for tilde P_*}. By Lemma \ref{convergence of the regularized system}, $\Tilde P^*_{r}(z_0)$ converges as $r\to\infty$ and hence $\Tilde P^*_{r}(z_0)$ is bounded in $r$. Thus, \eqref{diff equation for tilde P_*} and Lemma \ref{coefficients bound} give
\begin{gather*}
    \left|\frac{\d}{\d r}\Tilde{P}_r^*(z_0)\right|\ls\left(\sqrt{|\K'(r/2)|} + |\K'(r/2)|\right)\left|\Tilde{P}_r(z_0)\right| + |\K'(r/2)|.
\end{gather*}
If we integrate the latter inequality on $[r, +\infty)$ and apply Lemma \ref{convergence of the regularized system} together with the obtained bound for $\Tilde{P}_r(z_0)$, we will get the required bound for $\Tilde P^*_{r}(z_0) - \Pi(z_0)$. To conclude the proof of the corollary notice that Lemma \ref{lemma two repr kernels} yields
\begin{gather*}
    \int_{r}^{\infty} |P(x,z_0)|^2\,dx \ls \left(|\Pi(z_0)|^2 -|\Tilde P_{r}^*(z_0)|^2\right)  + |\Tilde P_{r}(z_0)|^2\ls e^{-\delta r/2}.
\end{gather*}
\end{proof}
%\section{Different construction of the analytical extension}\label{section Different construction of the analytical extension}
% \input{auxiliary lemma}
\subsection{Analytic continuation via the Christoffel-Darboux formula}\label{section Different construction of the analytical extension}

\begin{thm}\label{theorem C-D continuation}
Let $p,q\in L^1_{\loc}(\R_+)$ be real-valued functions and $Q = \left(\begin{smallmatrix}-q&p\\p&q\end{smallmatrix}\right)$.
Assume that there exists $\delta > 0$ such that $E_Q(r) = O\left(e^{-\delta r}\right)$ as $r\to\infty$. Fix an arbitrary number $h >\delta/4$. Then the integral \begin{gather}\label{infinite integral}
    \int_0^{\infty}P(x, z)\ol{P(x,ih)}\,dx
\end{gather}
converges uniformly on compact subsets of $\Omega_{\delta/4} = \{z\colon \Im z > -\delta/4\}$ and the function
\begin{gather}\label{extension of Pi via the C-D fomula on the line}
    z\longmapsto \frac{z + ih}{i\ol{\Pi(ih)}}\int_0^{\infty} P(x, z)\ol{P(x,ih)}\,dx, \quad z\in \Omega_{\delta/4}
\end{gather}
is analytic in $\Omega_{\delta/4}$ and coincides with $\Pi$ in $\Cm_+$.
\end{thm}
\begin{proof}
Substitute $ih$ into Christoffel-Darboux formula \eqref{first CD formula}. We have
\begin{gather}\label{temp Christ - Darb formula}
    i \frac{P_*(r, z)\ol{P_*(r, ih)} - P(r, z)\ol{P(r, ih)}}{z  + ih} = \int_0^r P(x, z)\ol{P(x,ih)}\,dx,\quad z\in \Cm.
\end{gather}
Take an arbitrary increasing sequence $\rho_n\to\infty$ such that $P(\rho_n, ih)\to 0$ as $n\to\infty$. Then from Lemma \ref{lemma on condition (c) of Krein Theorem} we know that there exist a subsequence $n_k$ and $\gamma\in [0,2\pi)$  such that
\begin{gather*}
    P(\rho_{n_k},z)\to 0,\qquad   P_*(\rho_{n_k},z)\to e^{i\gamma}\Pi(z),
\end{gather*}
uniformly on compact subsets of $\Cm_+$. Substituting $\rho_{n_k}$ for $r$ into \eqref{temp Christ - Darb formula} and taking the limit as $k\to\infty$, we obtain
\begin{gather*}
    \nonumber
    i \frac{\Pi(z)\ol{\Pi(ih)}}{z  + ih} = \lim_{k\to\infty}\int_0^{\rho_{n_k}} P(x, z)\ol{P(x,ih)}\,dx, \quad z\in \Cm_+,
\end{gather*}
or, equivalently,
\begin{gather}
    \label{pi continuation on subseq}
    \Pi(z) = \frac{z + ih}{i\ol{\Pi(ih)}}\lim_{k\to\infty}\int_0^{\rho_{n_k}} P(x, z)\ol{P(x,ih)}\,dx, \quad z\in \Cm_+.
\end{gather}
Therefore, the fact that \eqref{extension of Pi via the C-D fomula on the line} defines an analytic continuation of $\Pi$ immediately follows from the convergence of integral \eqref{infinite integral}. 
\medskip

For two positive real numbers $A \le  B$, define
\begin{gather*}
    F_{A, B}(z) = \int_A^{B}P(r, z)\ol{P(r,ih)}\,dr.
\end{gather*}
From Corollary \ref{corollary which was step 2} we have
\begin{gather}\label{start of step 3}
    \int_{r}^{\infty} |P(x,ih)|^2\,dx \ls e^{-\delta r/2},\quad r\ge 0.
\end{gather}
Then, for $z\in\Cm_+$, the Cauchy-Schwartz inequality gives 
\begin{align*}
    |F_{A, B}(z)|&\le \sqrt{\int_A^B|P(r, z)|^2\,dr}\sqrt{\int_A^B|P(r, ih)|^2\,dr}
    \\
    &\stackrel{\eqref{start of step 3}}{\ls} \cdot \sqrt{\int_0^{\infty}|P(r, z)|^2\,dr}\cdot e^{-A\delta/4}\stackrel{\eqref{expression for the szego function}}{\ls}
    e^{-A\delta/4}\frac{|\Pi(z)|}{\sqrt{\Im z}}.
\end{align*}
Therefore, we have 
\begin{align}
    \label{bound in the upper half plane}
    |F_{A,B}(z)| \ls e^{-A\delta_1}\frac{|\Pi(z)|}{\sqrt{\Im z}}, \qquad z\in\Cm_+,
\end{align}
uniformly in $\Cm_+$. Because of reflection formula \eqref{Krein system reflection formula},  $F_{A, B}$ admits the following representation:
\begin{gather*}
    F_{A, B}(z) = \int_A^Be^{izr}\ol{P_*(r, \ol{z})}\ol{P(r,ih)}\,dr.
\end{gather*}
 By definition, for $\Delta > 0$ put $\Omega_{\Delta}^- = \{z\colon   0 > \Im z > -\Delta\}$. 
Using the same Cauchy-Schwartz argument for $z\in \Omega_{\delta/4}^-$ as for $z\in \Cm_+$, we get
\begin{align}
    \nonumber
    |F_{A, B}(z)| 
    &\le \sqrt{\int_A^Be^{2\Im\ol{z} r}|P(r, ih)|^2\,dr}\sqrt{\int_A^B|P_*(r, \ol{z})|^2\,dr} 
    \\
    \label{first bound for F_AB}
    &\ls \frac{e^{-A(\delta/4 - \Im\ol{z})}}{\delta/4 - \Im\ol{z}}\cdot \sqrt{\int_A^B|P_*(r, \ol{z})|^2\,dr}.
\end{align}
The function $|P_*(r, \ol{z})|^2$ is not summable on $\R_+$, however, the integral on the finite segment can be estimated by \eqref{second CD formula} and \eqref{expression for the szego function}. In other words, we have
\begin{align*}
    \int_A^B\left| P_*(r,\ol{z})\right|^2dr &= \int_A^B\left( \left|P(r,\ol{z})\right|^2  + 2 \Im(\ol{z})\int_0^{r} \left|P(s,\ol{z})\right|^2ds \right)dr 
    \\ 
    &= \int_A^B \left|P(r,\ol{z})\right|^2dr + 2\Im(\ol{z})  \int_A^B\int_0^{r} \left|P(s,\ol{z})\right|^2ds\,dr 
    \\
    &\le \int_0^{\infty} \left|P(r,\ol{z})\right|^2dr + 2\Im(\ol{z})  \int_A^B\int_0^{\infty} \left|P(s,\ol{z})\right|^2ds\,dr
    \\
    &= \left(\frac{1}{2\Im(\ol{z})}+ B - A \right) 2\Im(\ol{z})  \int_0^{\infty} \left|P(s,\ol{z})\right|^2ds
    \\
    &\le \left(\frac{1}{2\Im(\ol{z})}+B - A \right) |\Pi(\ol{z})|^2.
\end{align*}
If, in addition, $|B - A|\le 1$ then, for $z\in \Omega^-_{\delta/4}$, we have
\begin{align*}
    \frac{1}{2\Im(\ol{z})}+B - A &\le \frac{1 + 2\Im(\ol{z})}{2\Im(\ol{z})} \ls \frac{1}{\Im(\ol{z})},
    \\
    \int_A^B\left| P_*(r,\ol{z})\right|^2dr &\ls \frac{|\Pi(\ol{z})|^2}{\Im(\ol{z})}.
\end{align*}
Substituting the latter bound into  \eqref{first bound for F_AB}, we get
\begin{gather}
    \label{bound in the lower half plane}
    |F_{A,B}(z)|\ls \frac{e^{-A(\delta/4 - \Im\ol{z})}}{\delta/4 - \Im\ol{z}}\cdot\frac{|\Pi(\ol{z})|}{\sqrt{|\Im z|}},\quad z\in \Omega^-_{\delta/4}, \quad B - A\le 1.
\end{gather}
Fix a connected compact set $K\subset \Omega_{\delta/4}$. Let us show that there exists a positive constant $\alpha$ depending on $K$ such that
\begin{gather}\label{desired exponential bound}
    |F_{A,B}(z)| \ls e^{-A\alpha},
\end{gather}
uniformly for $A\le B$ and $z\in K$. Three different situations are possible:
\begin{gather*}
    K\subset \Cm_+,
    \quad 
    K \subset \Omega_{\delta/4}\cap \Cm_-,
    \quad
    K\cap \R \neq \emptyset.
\end{gather*}
In the first and in the second situations bound \eqref{desired exponential bound} for $B - A \le 1$ easily follows from \eqref{bound in the upper half plane} and \eqref{bound in the lower half plane} respectively.
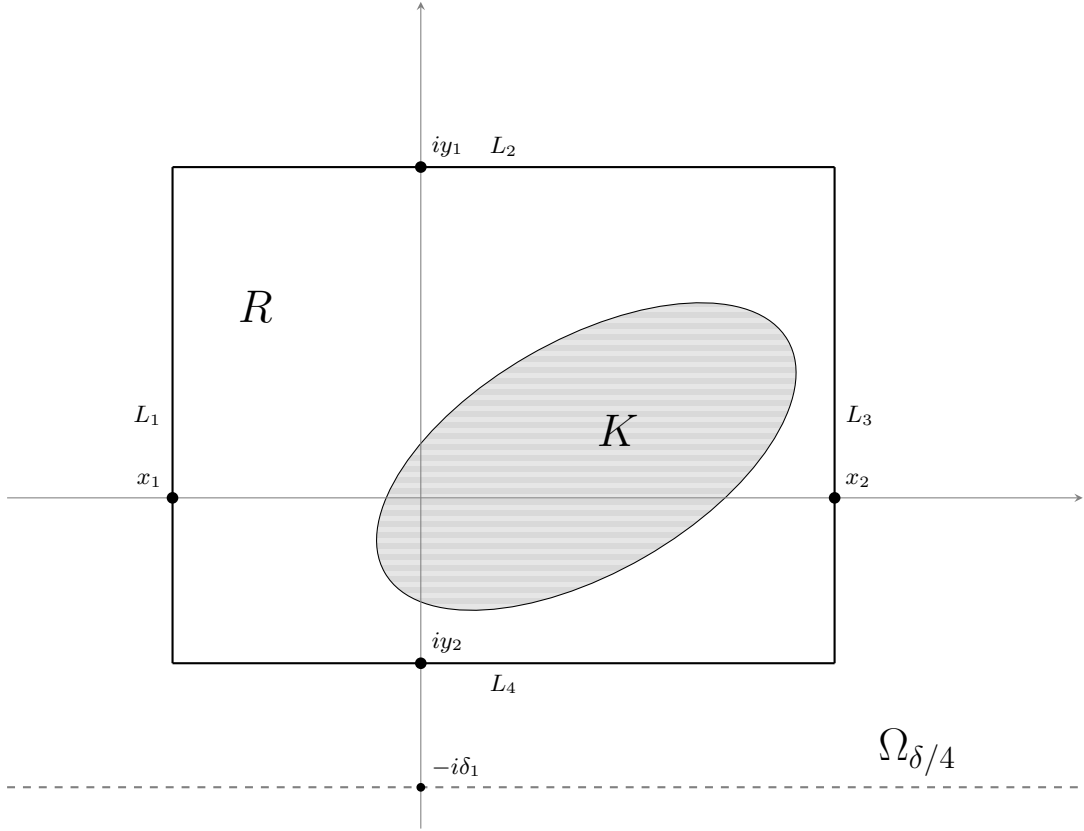
\begin{figure}[ht]
\centering
\tikzset{every node/.style={font=\scriptsize}}
    \begin{tikzpicture}[scale=1.1][every node/.append style={midway}]
    
    \coordinate (E) at (2,0.5);
    \draw[rotate around={-60:(E)}, pattern=horizontal lines light gray,label=above right:$K$] (E) ellipse (40pt and 80pt);
    \coordinate[label=above right:\Large $K$] (E) at (E);
    
    \coordinate (O) at (0,0);
    \coordinate (A) at (-3,-2);
    \coordinate (B) at (5,-2);
    \coordinate (C) at (5,4);
    \coordinate (D) at (-3,4);
    \coordinate (X) at (10,0);
    \draw[black, thick] (A) -- node[below] {$L_4$} (B);
    \draw[black, thick] (B) -- node[right] {$L_3$} (C);
    \draw[black, thick] (C) -- node[above] {$L_2$} (D);
    \draw[black, thick] (D) -- node[left] {$L_1$} (A);
    \draw [-stealth, gray](-5,0) -- (8,0);
    \draw [-stealth, gray](0,-4) -- (0,6);

    \coordinate (x1) at (-3,0);
    \coordinate[label=above left:$x_1$] (x1) at (x1);
    \fill[black] (x1) circle (2pt);
    
    \coordinate (x2) at (5,0);
    \coordinate[label=above right:$x_2$] (x2) at (x2);
    \fill[black] (x2) circle (2pt);
    
    \coordinate (y1) at (0,4);
    \coordinate[label=above right:$iy_1$] (y1) at (y1);
    \fill[black] (y1) circle (2pt);
    
    \coordinate (y2) at (0,-2);
    \coordinate[label=above right:$iy_2$] (y2) at (y2);
    \fill[black] (y2) circle (2pt);
    
    \draw[gray, dashed, thick] (-5, -3.5) -- (8, -3.5);
    
    \coordinate[label=\Large $R$] (-2, 2) at (-2, 2);
    
    \coordinate[label=\Large $\Omega_{\delta/4}$] (6, -3.5) at (6, -3.5);

    \coordinate (Deltah) at (0, -3.5);
    \coordinate[label=above right:$-i\delta_1$] (Deltah) at (Deltah);
    \fill[black] (Deltah) circle (1.5pt);
    
    \end{tikzpicture}
    \caption{ a compact $K$ and a rectangle $R$ in the proof of \eqref{desired exponential bound}}
    \label{fig:rectangle}
\end{figure}
If $K$ intersects the real line then take a rectangle $R$ with sides parallel to the real and imaginary axis of the complex plane such that 
$ K\subset R\subset \Omega_{\delta/4}$ and $ \dist(K, \d R) > 0$,
see Figure \ref{fig:rectangle}. Let $L_1, L_2, L_3$ and $L_4$ be the left, top, right and bottom sides of $R$ respectively and let $x_1, x_2, y_1, y_2$ be such that
\begin{gather*}
    L_1\subset \{z\colon \Re z = x_1\}, \qquad L_3\subset \{z\colon \Re z = x_2\},
    \\
    L_2\subset \{z\colon \Im z = y_1\}, \qquad L_4\subset \{z\colon \Im z = y_2\}.
\end{gather*}
By Lemma \ref{entropy and absolute continuity}, $\Pi$ is continuous in $\ol{\Cm_+}$ hence
\begin{gather}\label{boundedness of Pi on the boundary}
    \sup_{z\in\d R\cap \Cm_+} |\Pi(z)| < \infty,\quad \sup_{z\in\d R\cap \Cm_-} |\Pi(\ol{z})| < \infty.
\end{gather}
Denote $\delta/4 - |y_2| > 0$ by $\alpha$. Then, combining \eqref{bound in the upper half plane}, \eqref{bound in the lower half plane} and \eqref{boundedness of Pi on the boundary}, we obtain
\begin{gather}\label{bound on the boundary}
    |F_{A,B}(z)|\ls \frac{e^{-A\alpha}}{\sqrt{|\Im z|}},\quad z\in\d R\setminus \R, \quad B - A\le 1.
\end{gather}
For every $z_0\in K$ we have 
\begin{gather*}
    F_{A,B}(z_0) = \frac{1}{2\pi i}\int_{\d R} \frac{F_{A,B}(z)}{z - z_0}dz = \frac{1}{2\pi i}\sum_{n  = 1}^4 \int_{L_n} \frac{F_{A,B}(z)}{z - z_0}dz.
\end{gather*}
The inequality $|z - z_0| \ge \dist(K, \d R) > 0$ holds therefore 
\begin{gather}\label{Cauchy integration formula}
    | F_{A,B}(z_0)|\ls \int_{\d R} |F_{A,B}(z)||dz| = \sum_{n  = 1}^4 \int_{L_n} |F_{A,B}(z)||dz|.
\end{gather}
It remains to bound the integrals over the sides of $R$. We have
\begin{align*}
    \int_{L_2} |F_{A,B}(z)||dz| & =\int_{x_1}^{x_2}|F_{A,B}(x + iy_1)|dx \stackrel{\eqref{bound on the boundary}}{\ls} e^{-A\alpha},
    \\
    \int_{L_1} |F_{A,B}(z)||dz|  &=\int_{y_2}^{y_1}|F_{A,B}(x_1 + iy)|dy \stackrel{\eqref{bound on the boundary}}{\ls} e^{-A\alpha}\int_{y_2}^{y_1}\frac{1}{\sqrt{|\Im y|}}dy \ls e^{-A\alpha}. 
\end{align*}
The integrals over the segments $L_3$ and $L_4$ can be bounded similarly. Therefore 
\eqref{desired exponential bound} holds for $B - A\le 1$. If $ n \le B - A < n + 1$ for some positive integer $n$, then we have
\begin{gather*}
    |F_{A, B}(z)|\le \sum_{k = 0}^n |F_{A + k, A + k + 1}(z)| + |F_{A + n, B}(z)| \ls \sum_{k = 0}^n e^{-(A + k)\alpha} + e^{-(A + n)\alpha} \approx e^{-A\alpha},
\end{gather*}
hence \eqref{desired exponential bound} holds for all $A\le B$ with some other constant in $\ls$. The convergence of integral \eqref{infinite integral} follows and the proof of Theorem \ref{theorem C-D continuation} is finished.

\end{proof}
\section{Square summable potentials. Proof of Theorem \ref{theorem on equivalence}}\label{proof of Theorem 2}
In Theorem \ref{theorem on equivalence} the entries $p$ and $q$ of the potential $Q$ are in $L^2({\R_+})$. The coefficient $a$ of the corresponding Krein system (recall Section \ref{section Connections with the Dirac operator}) is also in $L^2({\R_+})$. As usual $\sigma$ is the spectral measure of the Dirac operator and of the Krein system, $\Pi$ is the inverse \Szego function of $\sigma$ and $P, P_*$ are the solutions of the Krein system with coefficient $a$. Let us start with a general result concerning Krein systems with $a\in L^2({\R_+})$.
%\subsection{Spectral theory of Krein system with $L^2$ coefficient}
\begin{knownthm}[S. Denisov]\label{theorem spectral theory of Dirac operator with L^2 potential}
Assume that the coefficient $a$ of Krein system \eqref{Krein system} belongs to $L^2(\R_+)$. Then $\sigma$ belongs to the \Szego class and there exists $\gamma\in[0, 2\pi)$ such that
\begin{gather*}
    P_*(r, \lambda)\to \Pi_{\gamma}(\lambda) = e^{i\gamma}\Pi(\lambda), \quad r\to\infty,
    \\
    |P_*(r,\lambda)|\approx 1, \quad |\Pi(\lambda)|\approx 1, \quad |P(r,\lambda)|\ls 1
\end{gather*}
uniformly in $\{\Im\lambda > \eps\}$ for every $\eps > 0$. Additionally, we have
\begin{gather*}
    \int_{\R}\left|\frac{1}{\Pi_{\gamma}(x)} - 1\right|^2dx < \infty.
\end{gather*}
\end{knownthm}
\begin{proof}
The \Szego condition and the convergence in $\{\Im\lambda > \eps\}$ is stated in Theorem 11.1 from \cite{Denisov2006}; the assertion $\Pi_{\gamma}^{-1} - 1\in L^2(\R_+)$ is Theorem 11.2 in \cite{Denisov2006}. The boundedness of $P_*$ and $P_*^{-1}$ follows from the proof of Theorem 11.1 and Lemma 4.6 in \cite{Denisov2006} respectively. Christoffel-Darboux formula \eqref{second CD formula} gives $|P(r, z)| < |P_*(r,z)|$ for $z\in\Cm_+$. Finally, the convergence of $P_*$ to $\Pi_{\gamma}$ implies the inequalities for $\Pi$ and $\Pi^{-1}$.
\end{proof}

\subsection{Auxiliary results}
\begin{lemma}\label{lemma bound on Pi with alpha and beta}
Assume that $a\in L^2(\R_+)$ and $E_Q(r) = O\left(e^{-\delta r}\right)$ as $r\to\infty$ for some $\delta > 0$. Then $\Pi$ is analytic in $\Omega_{\delta/4}$ and for every $\alpha \in (0,1)$ the inequality $|\Pi(z)| \ls (1 + |z|)^{\alpha}$ holds uniformly in the closed half-plane $\ol{\Omega_{\alpha\delta/4}}$.
\end{lemma}
\begin{proof}
The analyticity of $\Pi$ immediately follows from Theorem \ref{Entropy for extension}, hence it remains to prove the bound $(1 + |z|)^{\alpha}$.
To simplify the exposition introduce $\delta_1 = \delta/4$. Fix an arbitrary $h > \delta_1$. By Theorem \ref{theorem spectral theory of Dirac operator with L^2 potential}, we have $P_*(r, ih)\to \Pi_{\gamma}(ih)$ as $r\to\infty$. An application of the Cauchy–Schwarz inequality to \eqref{Krein system} gives
\begin{align*}
    |P_*(r,ih) - \Pi_{\gamma}(ih)| &= \left| \int_{r}^{\infty} a(x) P(x,ih) \, dx \right| \le\norm[a]_{L_2(\R_+)}\cdot\sqrt{\int_{r}^{\infty}|P(x, ih)|^2\,dx}.
\end{align*}
Corollary \ref{corollary which was step 2} yields
\begin{gather}\label{exp fast convergence of P star to Pi}
    |P_*(r,ih) - \Pi_{\gamma}(ih)| \ls e^{-r\delta_1}, \quad r\ge 0.
\end{gather}
Reordering the terms in \eqref{temp Christ - Darb formula}, we get
\begin{gather}
    \label{P star from C-D}
    P_*(r,z) = \frac{z  + ih}{i\ol{P_*(r, ih)}}\int_0^r P(x, z)\ol{P(x,ih)}\,dx + \frac{P(r, z)\ol{P(r, ih)}}{\ol{P_*(r, ih)}},\quad z\in \Cm.
\end{gather}
On the other hand, in Theorem \ref{theorem C-D continuation} we showed \eqref{extension of Pi via the C-D fomula on the line}  that
\begin{gather*}
    \Pi_{\gamma}(z) =  \frac{z + ih}{i\ol{\Pi_{\gamma}(ih)}}\int_0^{\infty} P(x, z)\ol{P(x,ih)}\,dx, \quad z\in \Omega_{\delta_1}.
\end{gather*}
The two latter equalities together give 
\begin{align*}
    \Pi_{\gamma}(z) - P_*(r,z) &= \frac{z + ih}{i\ol{\Pi_{\gamma}(ih)}}\int_r^{\infty} P(x, z)\ol{P(x,ih)}\,dx - \frac{P(r, z)\ol{P(r, ih)}}{\ol{P_*(r, ih)}}
    \\
    &+\frac{(z  + ih)}{i}\int_0^r P(x, z)\ol{P(x,ih)}\,dx\left({\frac{1}{\ol{\Pi_{\gamma}(ih)}} - \frac{1}{\ol{P_*(r, ih)}}}\right), \quad z\in \Omega_{\delta_1} .  
\end{align*}
Let $z$ be a point with $\Im z = -\alpha\delta_1$. From \eqref{Krein system reflection formula} we have
\begin{gather*}
    |P(r, z)| = |e^{irz}\ol{P_*(r,\ol{z})}| = e^{\alpha\delta_1r}|P_*(r,\ol{z})|.
\end{gather*}
The functions $P_*(r,\ol{z})$ and $P_*(r, ih)^{-1}$ are uniformly bounded for $r\ge 0$ by Theorem \ref{theorem spectral theory of Dirac operator with L^2 potential}. Therefore 
\begin{align*}
    |\Pi_{\gamma}(z) - P_*(r,z)| &\ls (|z| + 1)\int_r^{\infty} e^{\alpha\delta_1 x }|P(x,ih)|\,dx + e^{\alpha\delta_1r}|P(r,ih)|
    \\
    &+(|z| + 1)|\Pi_{\gamma}(ih) - P_*(r,ih)|\int_0^{r} e^{\alpha\delta_1 x}|P(x,ih)|\,dx,
\end{align*}
uniformly for $z$ with $\Im z = -\alpha\delta_1$.
Corollary \ref{corollary which was step 2} gives us that the integral in the first term is $O\left(e^{-(1-\alpha)\delta_1r}\right)$ as $r\to\infty$. For the same reason the integral in the last term is bounded for $r\ge 0$. Because of \eqref{exp fast convergence of P star to Pi}, the last term is $O\left(e^{-r\delta_1}\right)$ as $r\to\infty$. Therefore,
\begin{align*}
    |\Pi_{\gamma}(z) - P_*(r,z)| &\ls (|z| + 1)e^{-(1-\alpha)\delta_1r} +  e^{\alpha\delta_1r}|P(r,ih)|,\quad r\ge 0,
    \\
    \int_r^{r + 1}|\Pi_{\gamma}(z) - P_*(\rho,z)|\,d\rho &\ls (|z| + 1)e^{-(1-\alpha)\delta_1r} + \int_r^{r + 1}e^{\alpha\delta_1 \rho}|P(\rho,ih)|\,d\rho,\quad r\ge 0.
\end{align*}
Applying Corollary \ref{corollary which was step 2} for $ih$ one more time, we get
\begin{gather*}
    \int_r^{r + 1}|\Pi_{\gamma}(z) - P_*(\rho,z)|\,d\rho \ls (|z| + 1)e^{-(1-\alpha)\delta_1r}.
\end{gather*}
Therefore there exists a constant $C$ such that
\begin{gather*}
    \int_r^{r + 1}|P_*(\rho,z)|\,d\rho\ge |\Pi_{\gamma}(z)| - C(|z| + 1)e^{-(1-\alpha)\delta_1r},\quad r\ge 0.
\end{gather*}
Define $r_0 = r_0(z)$ as the solution of the equation
\begin{gather}\label{equation for r}
    e^{-\delta_1(1 - \alpha) r} = \frac{|\Pi_{\gamma}(z)|}{2C(|z| + 1)}.
\end{gather}
Then for every $r \ge r_1 = \min(r_0, 0)$ we have $ \int_r^{r + 1}|P_*(\rho,z)|\,d\rho \ge {|\Pi_{\gamma}(z)|}/2$. Formulas \eqref{expression for the szego function} and \eqref{Krein system reflection formula} for the point $\ol{z}\in\Cm_+$ and the latter inequality give
\begin{align*}
    |\Pi_{\gamma}(\ol{z})|^2 &= 2\Im \ol{z}\int_0^{\infty} |P(x,\ol{z})|^2\,dx \gs \int_{r_1}^{r_1 + 1} e^{-2\alpha\delta_1x}|P_*(x,z)|^2\,dx
    \gs |\Pi_{\gamma}(z)|^2e^{-2\alpha\delta_1 r_1}.
\end{align*}

From Theorem \ref{theorem spectral theory of Dirac operator with L^2 potential} we know that $\Pi_{\gamma}(\ol{z})$ is uniformly bounded for $\Im z =- \alpha\delta_1$ hence the latter implies $|\Pi_{\gamma}(z)|e^{-\alpha\delta_1 r_1} \ls 1$. If $r_1 = 0$ then we obtain $|\Pi_{\gamma}(z)|\ls 1$. If $r_1 = r_0$  we get
\begin{gather*}
    e^{-\alpha\delta_1 r_1} = e^{-\alpha\delta_1 r_0} = \left(e^{-\delta_1(1 - \alpha) r_0}\right)^{\beta} \approx \left(\frac{|\Pi_{\gamma}(z)|}{|z| + 1}\right)^{\beta},
\end{gather*}
where $\beta = \frac{\alpha}{1 - \alpha}$. Hence for $z$ with $\Im z =- \alpha\delta_1$  we have
\begin{gather}
    \nonumber
    |\Pi_{\gamma}(z)|\left(\frac{|\Pi_{\gamma}(z)|}{|z| + 1}\right)^{\beta}\ls 1,
    \\
    \label{bound for PI in the halfplane}
    |\Pi_{\gamma}(z)| \ls (|z| + 1)^{\frac{\beta}{1 + \beta}} = (|z| + 1)^{\alpha}.
\end{gather}
The function $z\mapsto \Pi_{\gamma}(z)(z + i\delta)^{-\alpha}$ is analytic in the strip $\{z\colon |\Im z| \le \alpha\delta_1\}$; by Theorem \ref{theorem spectral theory of Dirac operator with L^2 potential} and \eqref{bound for PI in the halfplane}, it is bounded on the boundary of this strip. Furthermore, from Corollary \ref{corollary bound for extended function Pi} it follows that this function grows no faster than the exponential function. Hence we can apply the Phragm\'{e}n – Lindel\"{o}f principle in the strip to deduce that $f$ is bounded in the strip. Consequently \eqref{bound for PI in the halfplane} holds in the whole closed half-plane $\ol{\Omega_{\alpha\delta_1}}$. The proof is finished.
\end{proof}

\begin{cor}\label{corollary from the entropy decay to the H2 }
Assume that $a\in L^2(\R_+)$ and $E_Q(r) = O\left(e^{-\delta r}\right)$ as $r\to\infty$ for some $\delta > 0$. Then $\displaystyle\frac{\Pi(x - i\Delta)}{x + i}\in H^2(\Cm_+)$ for every $\Delta < \delta/8$.
\end{cor}
\begin{proof}
Fix an arbitrary $\displaystyle \Delta < {\delta}/{8}$ and let $\displaystyle\alpha = {4\Delta}/{\delta} < {1}/{2}$. By Lemma \ref{lemma bound on Pi with alpha and beta}, the inequality $|\Pi(z)| \ls (1 + |z|)^{\alpha}$ holds in $\Omega_{\alpha\delta/4} = \Omega_{\Delta}$. We have $\alpha - 1 < -\frac{1}{2}$ hence $\Pi(1 + |z|)^{-1}$ is square integrable over horizontal lines in $\Omega_{\Delta}$.
\end{proof}
\begin{lemma}\label{lemma from extension of the decrease of P}
Assume that $a\in L^2(\R_+)$, $\sigma$ is absolutely continuous and $\displaystyle\frac{\Pi(x - i\Delta)}{x + i}\in H^2(\Cm_+)$ for some $\Delta > 0$. Then for every $z_0\in \Cm_+$ and $\delta < \min(\Delta, \Im z_0)$ we have
\begin{gather}
    \label{decay of P in integral sense}
    \int_r^{\infty}|P(r, z_0)|^2\,dr \ls e^{-2\delta r},\quad r\ge 0.
\end{gather}
\end{lemma}
\begin{proof}
Fix a point $z_0$ in $\Cm_+$ and $\delta  < \min(\Delta, \Im z_0)$. Denote $  {\Pi_{\gamma}}/{(z - \ol{z_0})}$ by $G$. We know that ${G(z - i\delta )\in H^2(\Cm_+)}$ hence there exists a function $\phi\in  L_2(\R_+)$, such that
\begin{gather*}
    G(z - i\delta) = \int_0^{\infty}\phi(t)e^{itz}\,dt,\quad z\in\Cm_+,
    \\
    G(z) =\int_0^{\infty}\phi(t)e^{-t\delta  }e^{itz }\,dt ,\quad z\in\Omega_{\delta}.
\end{gather*}
Let $G_r$ be the projection of $G$ onto the space $PW_{[0,r]}$ or, in other words, let
\begin{gather*}
    G_{r}(z) = \int_0^{r}\phi(t)e^{-t\delta  }e^{itz }\,dt,\quad z\in\Cm.
\end{gather*}
Recall definition \eqref{definition of the minimization function for Krein system} of the minimizing function $\bm_r$. We will show that $\bm_r(z_0)$ converges to $\bm_{\infty}(z_0)$ exponentially fast in $r$. We have $G_r\in PW_{[0,r]}$ therefore
\begin{gather}
    \label{bound for m_r in terms of G}
    \bm_r(\sigma,z_0)\le \frac{1}{2\pi}\norm[G_r/G_r(0)]_{L^2(\sigma)}= \frac{1}{2\pi|G_{r}(z_0)|^2}\int_{-\infty}^{\infty}\frac{|G_{r}(t)|^2}{|\Pi(t)|^2}\, dt.
\end{gather}
Let us examine the right hand side of the latter inequality. For $z\in\ol{\Cm_+}$, we have
\begin{gather*}
    G(z) - G_r(z) = \int_{r}^{\infty}\phi(t)e^{-t\delta  }e^{itz }\,dt.
\end{gather*}
This difference is uniformly bounded in $\ol{\Cm_+}$:
\begin{align}
    \nonumber
    \left|G(z) - G_r(z)\right| &\le \int_{r}^{\infty}\left|\phi(t)e^{-t\delta  }e^{itz }\right|\,dt = \int_{r}^{\infty}\left|\phi(t)\right|e^{-t(\Im z + \delta )}\,dt
    \\
    \label{difference of G  in the point }
    &\le  e^{-r(\Im z + \delta )}\int_{r}^{\infty}\left|\phi(t)\right|e^{-(t - r) \delta } \,dt \ls e^{-(\Im z + \delta )r},\quad r\ge 0,
\end{align}
where the last inequality follows from the Cauchy - Schwarz inequality. Hence
\begin{gather*}
    \left||G(z_0)|^2 - |G_r(z_0)|^2\right|\le \left|G(z_0) - G_r(z_0)\right|\left(|G(z_0)| + |G_r(z_0)|\right) \ls e^{-(\delta  + \Im z_0) r}.
\end{gather*}
Furthermore, $|G(z_0)|\neq 0$ therefore
\begin{gather}
    \nonumber
    \left|\frac{1}{|G_{r}(z_0)|^2} - \frac{1}{|G(z_0)|^2} \right| = \left| \frac{|G(z_0)|^2 - |G_{r}(z_0)|^2}{|G_{r}(z_0)|^2|G(z_0)|^2}   \right| \ls e^{-(\delta  + \Im z_0) r},
    \\
    \label{pointwise bound on difference G}
    \frac{1}{|G_{r}(z_0)|^2} \le \frac{1}{|G_(z_0)|^2} + O\left(e^{-(\delta  + \Im z_0) r}\right) = \frac{4(\Im z_0)^2}{|\Pi(z_0)|^2} + O\left(e^{-(\delta  + \Im z_0) r}\right),\quad r\to\infty.
\end{gather}
Let us estimate the integral in \eqref{bound for m_r in terms of G}. We have
\begin{align}
    \nonumber
    \int_{-\infty}^{\infty}\frac{|G_{r}(t)|^2}{|\Pi_{\gamma}(t)|^2}\, dt &= \int_{-\infty}^{\infty}\left|\frac{G(t)}{\Pi_{\gamma}(t)} + \frac{G_{r}(t) - G(t)}{\Pi_{\gamma}(t)}\right|^2\, dt
    = \int_{-\infty}^{\infty}\left|\frac{1}{t - \ol{z_0}} + \frac{G_{r}(t) - G(t)}{\Pi_{\gamma}(t)}\right|^2\, dt 
    \\
    \label{differences of G with respect to sigma}
    &= \int_{-\infty}^{\infty}\frac{1}{|t - \ol{z_0}|^2} + 2\Re\left(\frac{1}{t - z_0}\cdot\frac{G_{r}(t) - G(t)}{\Pi_{\gamma}(t)} \right) + \left|\frac{G_{r}(t) - G(t)}{\Pi_{\gamma}(t)}\right|^2\, dt. 
\end{align}
Define the set $S = \left\{t\colon\left|\frac{1}{\Pi_{\gamma}(t) } - 1\right| \le \frac{1}{2}\right\}\subset \R$. Then $|\Pi_{\gamma}|^{-1}< 2$ on $S$ and 
\begin{gather}
    \label{second integral on S}
    \int_{S}\left|\frac{G_{r}(t) - G(t)}{\Pi_{\gamma}(t)}\right|^2\, dt \le 4\int_{\R}\left|G_{r}(t) - G(t)\right|^2 \approx \norm[e^{-\delta  t}\phi\textbf{1}_{[r,\infty)}]^2_{L_2(\R)} \ls e^{-2\delta  t}.
\end{gather}
On the other hand, on $\R \setminus S$ we have 
\begin{gather*}
    3\left|\frac{1}{\Pi_{\gamma}} - 1\right|\ge \left|\frac{1}{\Pi_{\gamma}} - 1\right| + 1\ge \frac{1}{|\Pi_{\gamma}|}.
\end{gather*}
Combining this with Theorem \ref{theorem spectral theory of Dirac operator with L^2 potential}, we get
\begin{gather}
    \nonumber
    \int_{\R \setminus S} \frac{1}{|\Pi_{\gamma}(t)|^2}\, dt \le 9\int_{\R} \left|\frac{1}{\Pi_{\gamma}(t) } - 1\right|^2\, dt < \infty,
    \\
    \label{second integral on R setminus S}
    \int_{\R\setminus S}\left|\frac{G_{r}(t) - G(t)}{\Pi_{\gamma}(t)}\right|^2\, dt \le \sup_{t\in \R}|G_{r}(t) - G(t)|^2\int_{\R \setminus S} \frac{1}{|\Pi_{\gamma}(t)|^2}\, dt \stackrel{\eqref{difference of G  in the point }}{\ls}  e^{-2\delta  r}.
\end{gather}
Next, consider the integral 
\begin{gather*}
    I = \int_{-\infty}^{\infty}\frac{1}{t - \ol{z_0}}\cdot\frac{G_{r}(t) - G(t)}{\Pi_{\gamma}(t)} \, dt.
\end{gather*}
We know that both $G_r - G$ and $\displaystyle\frac{1}{(z - \ol{z_0})\Pi_{\gamma}}$  belong to $H^2(\Cm_+)$. Therefore $\displaystyle\frac{G_r - G}{(z - \ol{z_0})\Pi_{\gamma}}\in H^1(\Cm_+)$ and consequently $I = 0$. Rewrite the second term in  \eqref{differences of G with respect to sigma} in the following way:
\begin{align}
    \nonumber
    \int_{-\infty}^{\infty}\frac{1}{t - z_0}\cdot\frac{G_{r}(t) - G(t)}{\Pi_{\gamma}(t)} \, dt &= \int_{-\infty}^{\infty}\frac{1}{t - z_0}\cdot\frac{G_{r}(t) - G(t)}{\Pi_{\gamma}(t)} \, dt - I
    \\
    \label{first integral without I}
    &= -2i\Im z_0\int_{-\infty}^{\infty}\frac{1}{(t - z_0)(t - \ol{z_0})}\cdot\frac{G_{r}(t) - G(t)}{\Pi_{\gamma}(t)} \, dt.
\end{align}
Define 
\begin{gather*}
    F(z) = \frac{G_{r}(z) - G(z)}{(z - z_0)(z - \ol{z_0})\Pi_{\gamma}(z)},\quad z\in \ol{\Cm_+}.
\end{gather*}
Let us calculate $\int_{\R}F(t)\,dt$ using the contour integration. Denote by $x_0$ and $y_0$ the real and imaginary parts of $z_0$ and put $z_1 = \frac{1}{2}(z_0 + x_0)$. The  contour $C_R$ is shown at at Figure \ref{fig:contour}; it consists of the horizontal segment $J_R = [x_0 - R, x_0 + R]$, two vertical segments $L_R^{(1)} = [x_0 +R, z_1 + R]$, $L_R^{(2)} = [z_1 - R, x_0 - R]$ and the semicircle $A_R$ with center $z_1$ of radius $R$.
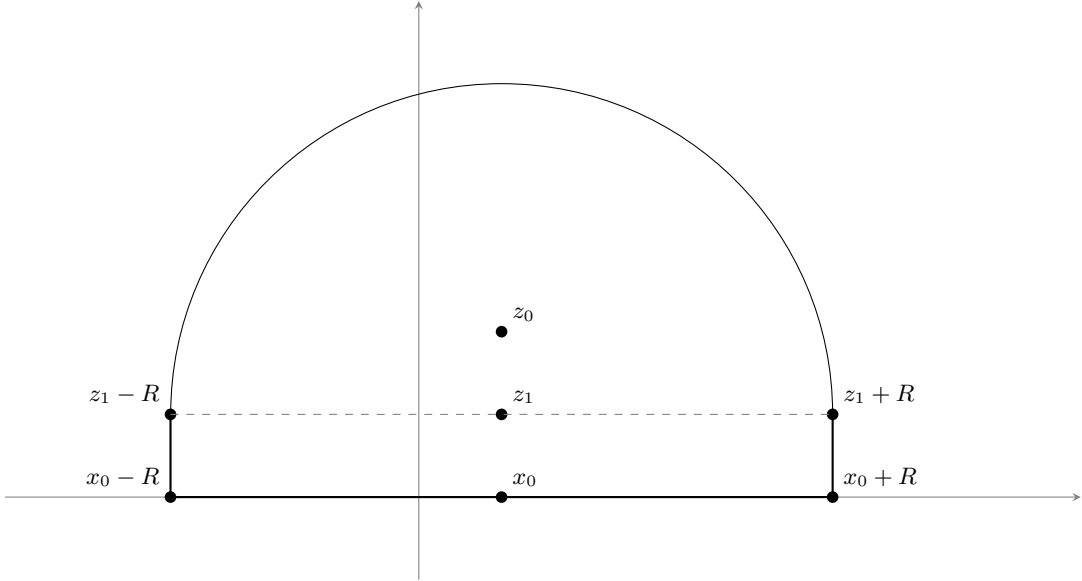
\begin{figure}[ht]
\centering
\tikzset{every node/.style={font=\scriptsize}}
    \begin{tikzpicture}[scale=1.1][every node/.append style={midway}]

    \draw [-stealth, gray](-5,0) -- (8,0);
    \draw [-stealth, gray](0,-1) -- (0,6);

    \coordinate (x1) at (-3,0);
    \coordinate[label=above left:$x_0 - R$] (x1) at (x1);
    \fill[black] (x1) circle (2pt);

    \coordinate (x2) at (5,0);
    \coordinate[label=above right:$x_0 + R$] (x2) at (x2);
    \fill[black] (x2) circle (2pt);
    
    \coordinate (x3) at (1, 0);
    \coordinate[label=above right:$x_0$] (x3) at (x3);
    \fill[black] (x3) circle (2pt);
    
    \coordinate (x4) at (-3,1);
    \coordinate[label=above left:$z_1- R$] (x4) at (x4);
    \fill[black] (x4) circle (2pt);
    
    \coordinate (x5) at (5,1);
    \coordinate[label=above right:$z_1 + R$] (x5) at (x5);
    \fill[black] (x5) circle (2pt);
    
    \coordinate (x6) at (1,1);
    \coordinate[label=above right:$z_1$] (x6) at (x6);
    \fill[black] (x6) circle (2pt);
    
    \coordinate (z_0) at (1, 2);
    \coordinate[label=above right:$z_0$] (z_0) at (z_0);
    \fill[black] (z_0) circle (2pt);
    
    \draw[black, thick] (x1) -- node[below] {$ $} (x4);
    \draw[black, thick] (x2) -- node[right] {$ $} (x5);
    \draw[black, thick] (x1) -- node[above] {$ $} (x2);
    \draw (x5) arc (0:180:4) ;
    
    \draw [dashed, gray](x4) -- (x5);
    
    \end{tikzpicture}
    \caption{ an integration contour}
    \label{fig:contour}
\end{figure}

The only singularity of $F$ inside $C_R$ is a simple pole at $z_0$ hence
\begin{gather*}
    \int_{C_R}F(z)dz = 2\pi i\Res_{z_0}F(z) = \frac{2\pi(G_{r}(z_0) - G(z_0))}{2\Im z_0 \Pi_{\gamma}(z_0)}.
\end{gather*}

By Theorem \ref{theorem spectral theory of Dirac operator with L^2 potential} and \eqref{difference of G  in the point }, the function $\frac{G_r - G}{\Pi_{\gamma}}$ is uniformly bounded in $\Im z > y_0/2$. Therefore
\begin{gather}
    \label{integral over the half circle}
    \left|\int_{A_R}F(z)\, dz\right|\ls \int_{A_R}\frac{1}{|(z - z_0)(z - \ol{z_0})|}|dz| = O\left(\frac{1}{R}\right),\quad R\to\infty.
\end{gather}
Next, consider $\displaystyle V(R) = \int_{L^{(1)}_R} F(z) \,dz + \int_{L^{(2)}_R} F(z) \,dz $. We have
\begin{gather*}
    |V(R)|^2 \ls\int_0^{\frac{y_0}{2}}|F(x_0 - R + iy)|^2dt + \int_0^{\frac{y_0}{2}}|F(x_0 + R + iy)|^2dy.
\end{gather*}
Therefore 
\begin{gather*}
    \int_{\R_+} |V(R)|^2 dR\le \iint_M |F(x + iy)|^2\,dx\,dy,
\end{gather*}
where $M = \{z\colon 0\le \Im z \le \frac{y_0}{2}\} = \R\times [0, \frac{y_0}{2}]$. For $z = x + iy\in M$ we have
\begin{gather*}
    |F(z)|\le \sup_{w\in M}\left|\frac{G_r(w) - G(w)}{w - z_0}\right|\cdot\frac{1}{|(z - z_0)\Pi_{\gamma}(z)|}\stackrel{\eqref{difference of G  in the point }}{\ls}\frac{1}{|(z - z_0)\Pi_{\gamma}(z)|}.
\end{gather*}
Combining the two latter inequalities, we get
\begin{gather}
\label{function V belongs to the L2}
    \norm[V]^2_{L^2(\R_+)}\le \int_{\R}\int_0^{\frac{y_0}{2}} \left|\frac{1}{(z - \ol{z_0})\Pi_{\gamma}(z)}\right|^2\,dy\,dx \ls \norm[\frac{1}{(z - \ol{z_0})\Pi_{\gamma}}]_{H^2(\Cm_+)}^2 < \infty.
\end{gather}
We have
\begin{gather*}
    \frac{2\pi(G_{r}(z_0) - G(z_0))}{2\Im z_0 \Pi(z_0)} = \int_{C_R} F(z) \,dz = V(R) + \int_{A_R}F(z)\, dz + \int_{J_R}F(z)\, dz,
    \\
    \left|\int_{J_R}F(z)\, dz - \frac{2\pi(G_{r}(z_0) - G(z_0))}{2\Im z_0 \Pi(z_0)}\right| \le |V(R)| + \left|\int_{A_R}F(z)\, dz\right|.
\end{gather*}
Because of \eqref{integral over the half circle} the second term tends to $0$ as $R\to\infty$; \eqref{function V belongs to the L2} implies that $\liminf\limits_{R\to\infty}\left|V(R)\right| = 0$. Thus, we have
\begin{gather}
    \label{asymptotics of the first integral}
    \left|\int_{\R}F(z)\, dz\right|  = \left|\frac{2\pi(G_{r}(z_0) - G(z_0))}{2\Im z_0 \Pi(z_0)}\right| \stackrel{\eqref{difference of G  in the point }}{\ls} e^{-(\delta + \Im z_0)r}.
\end{gather}
Substitution of \eqref{second integral on S}, \eqref{second integral on R setminus S}, \eqref{first integral without I} and \eqref{asymptotics of the first integral} into \eqref{differences of G with respect to sigma} gives
\begin{gather}
    \label{integral of G_r with respect to sigma}
    \int_{-\infty}^{\infty}\frac{|G_{r}(t)|^2}{|\Pi(t)|^2}\, dt =  \int_{-\infty}^{\infty}\frac{1}{|t - \ol{z_0}|^2}\, dt  + O\left(e^{-2\delta r}\right) = \frac{\pi}{\Im z_0} + O\left(e^{-2\delta r}\right),\quad r\to\infty.
\end{gather}

Substituting \eqref{pointwise bound on difference G} and \eqref{integral of G_r with respect to sigma} into \eqref{bound for m_r in terms of G} we get
\begin{align*}
    \bm_r(z_0) 
    &=  \left(\frac{4(\Im z_0)^2}{|\Pi(z_0)|^2} + O\left(e^{-(\delta  + \Im z_0) r}\right) \right)\left(\frac{1}{2\Im z_0} +O\left(e^{-2\delta  r}\right) \right)
    \\
    & = \frac{2\Im z_0}{|\Pi(z_0)|^2} + O\left(e^{-2\delta  r}\right)\stackrel{\eqref{m infinity formula} }{=}\bm_{\infty}(z_0) + O\left(e^{-2\delta  r}\right),\quad r\to\infty.
\end{align*}
To conclude the proof, rewrite $\bm_r$ and $\bm_{\infty}$ using relations \eqref{m_r function formula} and \eqref{m infinity formula}:
\begin{align*}
    \bm_r(z_0) - \bm_{\infty}(z_0) &= \left(\int_0^{r}|P(r, z_0)|^2\,dr\right)^{-1} - \left(\int_0^{\infty}|P(r, z_0)|^2\,dr\right)^{-1}
    \gs\int_r^{\infty}|P(r, z_0)|^2\,dr.
\end{align*}
Inequality \eqref{decay of P in integral sense} immediately follows.
\end{proof}
\begin{cor}\label{corollary from extension to the decay of entropy}
Under the assumptions of Lemma \ref{lemma from extension of the decrease of P} for every $\delta < \min(4, 4 \Delta)$, we have $E_Q(r) =O\left(e^{-\delta r}\right) $ as $r\to\infty$.
\end{cor}
\begin{proof}
Differential equation \eqref{diff equation for modulus of P_*} for $z = i$ becomes
\begin{gather}\label{diff eqaution in the point i}
    \frac{\d}{\d r}|\Tilde{P}_r^*(i)|^2  = -2\Im i\Re f_2(r) |\Tilde{P}_r^*(i)|^2 = \frac{|\K'(r/2)|}{2}|\Tilde{P}_r^*(i)|^2.
\end{gather}
We see that $|\Tilde{P}_r^*(i)|^2$ is increasing and Lemma \ref{convergence of the regularized system} gives $\displaystyle|\Tilde{P}_r^*(i)|^2 \le \lim_{r\to\infty} |\Tilde{P}_r^*(i)|^2 =  |\Pi(i)|^2$.
Rewrite the equality in Lemma \ref{lemma two repr kernels} at the point $z = i$ in a form
\begin{gather*}
    2\int_{r}^{\infty} |P(x,i)|^2\,dx = \left(|\Pi(i)|^2 - |\Tilde P_{r}^*(i)|^2\right)  + |\Tilde P_{r}(i)|^2 .
\end{gather*}
Both terms in the right hand side are nonnegative therefore, from Lemma \ref{lemma from extension of the decrease of P}, it follows that
\begin{gather}\label{fast convergence in the end}
    0\le |\Pi(i)|^2 - |\Tilde{P}_r^*(i)|^2 \ls e^{-2\delta r}
\end{gather}
holds for every $\delta < \min(1, \Delta)$. On the other hand,   $|\Tilde{P}_r^*(i)|\gs 1$ for $r\ge 0$ and   \eqref{diff eqaution in the point i} yields
\begin{align*}
    |\Pi(i)|^2 - |\Tilde{P}_r^*(i)|^2 = \int_r^{\infty}\frac{\d}{\d r}|\Tilde{P}_s^*(i)|^2 \, ds= \int_r^{\infty}\frac{|\K'(s/2)|}{2}|\Tilde{P}_s^*(i)|^2 \, ds\gs \K(r/2). 
\end{align*}
This together with \eqref{fast convergence in the end} and Theorem \ref{known theorem on two entropies} finishes the proof:
\begin{gather*}
    \K(r/2) \ls e^{-2\delta r},\qquad \K(r) \ls e^{-4\delta r},\qquad E_Q(r)\ls e^{-4\delta r}.
\end{gather*}
\end{proof}

\begin{rem}
The restrictive assertion $\delta < 4$ is imposed because the point $z = i$ is special for differential equation \eqref{diff equation for modulus of P_*}. At any other point the term $2\Re\left((z - i)f_1(r)\Tilde{P}_r(z)\ol{\Tilde{P}_r^*(z)}\right)$ does not vanish preventing further argumentation. To overcome this obstacle, below we introduce the scaled entropy and do the rescaling procedure.

\end{rem}
\subsection{Rescaling}\label{section rescaling}
%In this subsection we will show how to change the assertion  $\delta < \min(4, 4 \Delta_{\Pi})$ of Corollary \ref{corollary from extension to the decay of entropy} to $\delta <4 \Delta_{\Pi}$. \textbf{Что-то написать}
\subsubsection{Auxiliary estimates.}
For a segment $I$ and a measurable function $f$ on $I$ we use the notation $\langle f \rangle_I$ to denote the average of $f$ over $I$, i.e.,
\begin{gather*}
    \langle f \rangle_I = \frac{1}{|I|}\int_I f(x)\, dx.
\end{gather*}
We fix some segment $I = [a,b]$ and define
\begin{gather*}
    I_1 = [a,(a + b)/2], \quad I_2 = [(3a + b)/4, (a + 3b)/4], \quad I_3 = [(a + b)/2, b]
\end{gather*}
-- the left, the middle and the right  half of $I$ respectively. This notation is used throughout the whole subsection.
\begin{lemma}\label{lemma inverse CBS inequality}
Let $g$ be a positive measurable function on $I$. Then 
\begin{gather*}
      \langle g\rangle_{I} \langle g^{-1}\rangle_{I} - 1 \ls \max_{n\in\{1,2,3\}}\left[\langle g\rangle_{I_n} \langle g^{-1}\rangle_{I_n} - 1\right],
\end{gather*}
provided that the latter maximum does not exceed $1$.
\end{lemma}
\begin{proof}
Denote the maximum from the statement of the lemma by $\eps$. We have
\begin{align*}
    \langle g\rangle_{I} \langle g^{-1}\rangle_{I} &=\frac{1}{4}\left( \langle g\rangle_{I_1} + \langle g\rangle_{I_3}\right)\left( \langle g^{-1}\rangle_{I_1} + \langle g^{-1}\rangle_{I_3}\right)
    \\
    &= \frac{1}{4} \left( \langle g\rangle_{I_1} - \langle g\rangle_{I_3}\right)\left( \langle g^{-1}\rangle_{I_1} - \langle g^{-1}\rangle_{I_3}\right) + \frac{1}{2}\left(\langle g\rangle_{I_1}\langle g^{-1}\rangle_{I_1} + \langle g\rangle_{I_3}\langle g^{-1}\rangle_{I_3}\right).
\end{align*}
Therefore
\begin{align}
    \nonumber
    \langle g\rangle_{I} \langle g^{-1}\rangle_{I} - 1&\le \frac{\langle g\rangle_{I_1}\langle g^{-1}\rangle_{I_1} - 1}{2} + \frac{\langle g\rangle_{I_3}\langle g^{-1}\rangle_{I_3} - 1}{2}
    \\
    \nonumber
    &+ \frac{1}{4} \left|\left( \langle g\rangle_{I_1} - \langle g\rangle_{I_3}\right)\left( \langle g^{-1}\rangle_{I_1} - \langle g^{-1}\rangle_{I_3}\right)\right|
    \\
    \nonumber
     &\le \eps + \frac{1}{4}\langle g\rangle_{I_1}\langle g^{-1}\rangle_{I_1}\left|\left( 1 - \frac{\langle g\rangle_{I_3}}{\langle g\rangle_{I_1}}\right)\left( 1 - \frac{\langle g^{-1}\rangle_{I_3}}{\langle g^{-1}\rangle_{I_1}}\right)\right|
     \\
     \label{formula ratio of entropies on [0,2]}
     &\le \eps + \frac{1 + \eps}{4}\left|\left( 1 - \frac{\langle g\rangle_{I_3}}{\langle g\rangle_{I_1}}\right)\left( 1 - \frac{\langle g^{-1}\rangle_{I_3}}{\langle g^{-1}\rangle_{I_1}}\right)\right|.
\end{align}
Let $f = \log g$. By Jensen's inequality, for $n\in \{1,2,3\}$ we have
\begin{gather}
    \label{almost Jensen for exponents on three intervals}
    1\le \langle e^f\rangle_{I_n}e^{-\langle f\rangle_{I_n}} \le \langle e^f\rangle_{I_n} \langle e^{-f}\rangle_{I_n} =  \langle g\rangle_{I_n} \langle g^{-1}\rangle_{I_n}\le 1 + \eps.
\end{gather}
It follows that (see formula (4.12) in  \cite{Bessonov2022} or (3.7) in \cite{Korey1998})
\begin{gather*}
    \langle|f - \langle f\rangle_{I_n}|\rangle_{I_n} \ls \sqrt{\eps}.
\end{gather*}
As a corollary, we get
\begin{align}
    \nonumber
    \left|\langle f\rangle_{I_1} - \langle f\rangle_{I_2}\right| &= \langle\left|\langle f\rangle_{I_1} - \langle f\rangle_{I_2}\right|\rangle_{I_1\cap I_2} 
    \\
    \nonumber
    &\le \langle\left|\langle f\rangle_{I_1} - f\right|\rangle_{I_1\cap I_2} + \langle\left|f - \langle f\rangle_{I_2}\right|\rangle_{I_1\cap I_2}
    \\
    \label{small differences of averages}
    &\le \frac{1}{2} \langle\left|\langle f\rangle_{I_1} - f\right|\rangle_{I_1} + \frac{1}{2} \langle\left|f - \langle f\rangle_{I_2} \right|\rangle_{I_2} 
    \ls \sqrt{\eps}.
\end{align}
Inequalities \eqref{almost Jensen for exponents on three intervals} and \eqref{small differences of averages} yield
\begin{gather*}
    \left|\frac{\langle g\rangle_{I_1}}{\langle g\rangle_{I_2}} -1 \right| = \left|\frac{\langle e^f\rangle_{I_1}}{\langle e^f\rangle_{I_2}} -1 \right| = \left|\frac{\langle e^f\rangle_{I_1}}{e^{\langle f\rangle_{I_1}}} \cdot \frac{e^{\langle f\rangle_{I_1}}}{e^{\langle f\rangle_{I_2}}}\cdot \frac{e^{\langle f\rangle_{I_2}}}{\langle e^f\rangle_{I_2}} -1 \right| \ls \sqrt{\eps}.
\end{gather*}
Similarly we can obtain
\begin{gather*}
    \left|\frac{\langle g^{-1}\rangle_{I_1}}{\langle g^{-1}\rangle_{I_2}} -1 \right|\ls \sqrt{\eps}, \quad \left|\frac{\langle g \rangle_{I_2}}{\langle g \rangle_{I_3}} -1 \right|\ls \sqrt{\eps}, \quad \left|\frac{\langle g^{-1}\rangle_{I_2}}{\langle g^{-1}\rangle_{I_3}} -1 \right|\ls \sqrt{\eps}.
\end{gather*}
Therefore, we have
\begin{gather}
    \label{average exponential ratios}
    \left|\frac{\langle g \rangle_{I_3}}{\langle g \rangle_{I_1}} -1 \right|\ls\sqrt{\eps}, \quad \left|\frac{\langle g^{-1}\rangle_{I_3}}{\langle g^{-1}\rangle_{I_1}} -1 \right|\ls\sqrt{\eps}.
\end{gather}
The substitution of \eqref{average exponential ratios}  into \eqref{formula ratio of entropies on [0,2]} concludes the proof.
\end{proof}
\begin{lemma}\label{hamiltonian lemma}
Let $G$ be a matrix-valued function on $I$ such that $G(t)\ge 0$ and $\det G(t) = 1$ for almost every $t\in I$. Additionally assume that for every unit vector $v\in \R^2$ and $t\in I$ we have
\begin{gather}\label{hamilnonian lemma double inequality for quad form of H}
    \langle G(t)v, v\rangle \approx 1.
\end{gather}
Then the inequality
\begin{gather}
    \label{main inequality in hamiltonian lemma}
    \det\langle G\rangle_I - 1\ls \max_{n\in\{1,2,3\}}\left[\det\langle G\rangle_{I_n} - 1\right] 
\end{gather}
holds, provided that the latter maximum does not exceed $1$.
\end{lemma}
\begin{proof}
Denote the maximum from the statement of the lemma by $\eps$. For every $2\times 2$ matrix $M$ we have
\begin{gather*}
    \frac{1}{\sqrt{\det M}} = \frac{1}{\pi}\int_{\R^2} e^{-\langle Mx, x \rangle} dx.
\end{gather*}
Rewriting this in polar coordinates, we get
\begin{gather*}
    \frac{1}{\sqrt{\det M}} = \frac{1}{\pi}\int_\T\left(\int_0^{\infty}e^{-r^2\langle M\zeta, \zeta \rangle}r\,dr \right) d\zeta = \frac{1}{2\pi}\int_\T\frac{d\zeta}{\langle M\zeta, \zeta \rangle}.
\end{gather*}
Let $J$ be some subinterval of $I$ and define $A_J = \langle G\rangle_J$. By Jensen's inequality
\begin{align*}
    \frac{1}{\sqrt{\det A_J}} = \frac{1}{2\pi}\int_\T\frac{d\zeta}{\langle A_J\zeta, \zeta \rangle} &\le \frac{1}{2\pi}\int_\T\left\langle\frac{1}{\langle G\zeta, \zeta \rangle}\right\rangle_{J} d\zeta 
    \\ 
    &=\left\langle \frac{1}{2\pi}\int_\T\frac{d\zeta}{\langle G\zeta, \zeta \rangle}\right\rangle_{J}
    = \left\langle\frac{1}{\sqrt{\det G}}\right\rangle_{J} = 1.
\end{align*}
Therefore
\begin{gather*}
    0\le \int_\T\left(\left\langle\frac{1}{\langle G\zeta, \zeta \rangle}\right\rangle_{J} - \frac{1}{\langle A_J\zeta, \zeta \rangle}\right) d\zeta = 2\pi\left(1 - \frac{1}{\sqrt{\det A_J}}\right).
\end{gather*}
Denote $\langle G(t)\zeta, \zeta \rangle$ by $g_{\zeta}(t)$ then the latter equality can be rewritten as
\begin{gather*}
    \int_{\T} \langle g_{\zeta}^{-1}\rangle_{J} - \langle g_{\zeta}\rangle^{-1}_{J} \,d\zeta = 2\pi\left(1 - \frac{1}{\sqrt{\det A_J}}\right).
\end{gather*}
Next, we see that
\begin{gather*}
    \langle g_{\zeta}^{-1}\rangle_{J} - \langle g_{\zeta}\rangle^{-1}_{J} = \frac{\langle g_{\zeta}^{-1}\rangle_{J}\langle g_{\zeta}\rangle_{J} - 1}{\langle g_{\zeta}\rangle_{J}}\approx \langle g_{\zeta}^{-1}\rangle_{J}\langle g_{\zeta}\rangle_{J} - 1 
\end{gather*}
because of assumption \eqref{hamilnonian lemma double inequality for quad form of H}. Therefore
\begin{gather*}
    1 - \frac{1}{\sqrt{\det A_J}} \approx \int_{\T} F_J(\zeta)\,d\zeta,
\end{gather*}
where $F_J(\zeta) = \langle g_{\zeta}^{-1}\rangle_{J}\langle g_{\zeta}\rangle_{J} - 1$. Moreover, \eqref{main inequality in hamiltonian lemma} is equivalent to 
\begin{gather}
    \label{equivalent inequality in the hamiltonian lemma}
    \int_{\T}F_I(\zeta) \,d\zeta \ls \max_{n\in\{1,2,3\}}\left[\int_{\T}F_{I_n}(\zeta) \,d\zeta\right].
\end{gather}
% Two latter equalities give
% \begin{gather*}
%     \int_{\T}\langle g_{\zeta}^{-1}\rangle_{J}\langle g_{\zeta}\rangle_{J} - 1 \,d\zeta\approx 1 - \frac{1}{\sqrt{\det A_J}}.
% \end{gather*}
% Clearly, \eqref{main inequality in hamiltonian lemma} is then equivalent to 
% \begin{gather}
%     \label{equivalent inequality in the hamiltonian lemma}
%     \int_{\T}\langle g_{\zeta}^{-1}\rangle_{I}\langle g_{\zeta}\rangle_{I} - 1 \,d\zeta \ls \sup_{n\in\{1,2,3\}}\left[\int_{\T}\langle g_{\zeta}^{-1}\rangle_{I_n}\langle g_{\zeta}\rangle_{I_n} - 1 \,d\zeta\right].
% \end{gather}
% Define
Define the set $S\subset\T$ by
\begin{gather*}
    % F(\zeta) = \langle g_{\zeta}\rangle_{I} \langle g_{\zeta}^{-1}\rangle_{I} - 1,
    % \\
    % F_n(\zeta) = \langle g_{\zeta}\rangle_{I_n} \langle g_{\zeta}^{-1}\rangle_{I_n} - 1,\quad n = 1,2,3,
    % \\
    S = \{\zeta\colon F_n(\zeta) < 1,\;\; n = 1,2,3\}.
\end{gather*}
By Lemma \ref{lemma inverse CBS inequality}, for $\zeta\in S$ we have 
\begin{gather}\label{inequality by a sum of three}
    F(\zeta) \ls \max_{n\in\{1,2,3\}} F_n(\zeta).
\end{gather}
On the other hand, we have $\displaystyle \max_{n\in\{1,2,3\}} F_n(\zeta)\ge 1$ on $\T\setminus S$ and $F_I\ls 1$ on $\T\setminus S$ because of assumption \eqref{hamilnonian lemma double inequality for quad form of H}. Hence \eqref{inequality by a sum of three} holds on the whole circle with some other constant in $\ls$. Inequality \eqref{equivalent inequality in the hamiltonian lemma} follows and the proof is concluded.
\end{proof}

\subsubsection{Ordered exponent}
Let $A$ be a matrix-valued function with real entries in $L^1_{\loc}(\R_+)$. The solution of the differential equation 
\begin{gather}\label{differential equation of ordered exponent}
    \frac{\d }{\d t}X_A(r, t) = A(t) X_A(r, t), \quad X_A(r,r) = I,\quad t\ge r
\end{gather}
is called the ordered exponent of $A$. It follows from \eqref{differential equation of ordered exponent} that the relation
\begin{gather}
    \label{group property of ordered exponent}
    X_A(r, t)X_A(t_0, r) = X_A(t_0, t)
\end{gather}
holds for every $t_0 \le r \le t$. Moreover, $X_A$ admits the following explicit representation:
\begin{gather}
    \label{explicit representaion ordered exponent}
    X_A(r,t) = \sum_{m\ge 0}\int_r^t A(t_1) \int_r^{t_1} A(t_2)\int_r^{t_2}\ldots \int_r^{t_{m - 1}} A(t_m)\, dt_m\,\ldots dt_3\,dt_2\, dt_1,
\end{gather}
for details see the book \cite{Fried2002}. We remark that the series converges in the operator norm because the norm of the $m$-th term does nor exceed $\displaystyle \frac{1}{m!}\left(\int_r^t\norm[A(s)]\,ds\right)^m$.

Fix an arbitrary unit vector $v$ and consider
\begin{gather*}
    \frac{\partial}{\partial t}\norm[X_A(r,t)v]^2 = 2\left\langle\frac{\partial}{\partial t} X_A(r,t)v, X_A(r,t)v\right\rangle = 2\langle AX_A(r,t)v, X_A(r,t)v\rangle.
\end{gather*}
Hence 
\begin{gather*}
    2\norm[A(t)]\norm[X_A(r,t)v]^2\ge \frac{\partial}{\partial t}\norm[X_A(r,t)v]^2\ge - 2\norm[A(t)]\norm[X_A(r,t)v]^2,
    \\
    2\norm[A(t)] \ge \frac{\partial}{\partial t}\log \norm[X_A(r,t)v]^2 \ge-2\norm[A(t)].
\end{gather*}
Integrating on $[r, t]$ and taking the exponent, we get
\begin{gather}
    \label{norm inequality for ordered exponent}
    \exp\left[\int_{r}^t\norm[A(s)]\,ds\right]\ge \norm[X_A(r,t)v]\ge \exp\left[-\int_{r}^t\norm[A(s)]\,ds\right].
\end{gather}

\subsubsection{Entropy and rescaling.}
Recall definition \eqref{hamiltonian for dirac} of the entropy function. In this subsection we show what happens if we replace the integration segment $[r, r + 2]$ with $[r, r + 2l]$ in \eqref{hamiltonian for dirac} for different parameters $l > 0$. Namely, define the scaled entropy function by
\begin{gather}\label{definition of the scaled entropy}
    E_Q^{(l)}(r) = \frac{1}{l^2}\det\int_{r}^{r + 2l}  H_Q(t)\,dt - 4,
\end{gather}
so that $E_Q(r) = E_Q^{(1)}(r)$.
Equation \eqref{Dirac system} for $\lambda = 0$ can be rewritten as $N_Q'(t) = JQ(t)N_Q(t)$. This is a differential equation for the ordered exponent  of $JQ$ hence $N_Q(t) = X_{JQ}(0,t)$. Using \eqref{group property of ordered exponent} and \eqref{hamiltonian for dirac} for all $r < t$ we get
\begin{gather*}
    N_Q(t) = X_{JQ}(r, t)X_{JQ}(0, r),
    \\
    H_Q(t) = (X_{JQ}(0, r))^*(X_{JQ}(r, t))^*X_{JQ}(r, t)X_{JQ}(0, r).
\end{gather*}
For $r\ge 0$ we have $\trace(JQ(r)) = 0$ hence $\det X_{JQ}(0, r) = 1$. Therefore \eqref{definition of the scaled entropy} can be rewritten as 
\begin{gather}\label{another representation of scaled entropy}
    E_Q^{(l)}(r) = \frac{1}{l^2}\det\int_{r}^{r + 2l}  (X_{JQ}(r, t))^*X_{JQ}(r, t)\,dt - 4.
\end{gather}
\begin{thm}\label{theorem decay of different entropies}
Assume that $p,q\in L^2(\R_+)$ and let the scaled entropy $E_Q^{(l)}$ be defined by \eqref{another representation of scaled entropy}. Let $f$ be an arbitrary nonincreasing function with $\lim f(r) = 0$ as $r\to\infty$. If the assertion
\begin{gather}\label{decay of the scaled entropy in the lemma}
    E_Q^{(l)}(r) \ls f(r),\quad r\ge 0
\end{gather}
holds for some $l > 0$ then it holds for every $l > 0$.
\end{thm}
\begin{proof}
Assume that \eqref{decay of the scaled entropy in the lemma} holds for  $l=l_0$. Let us show that for all large $r$, the assertions of Lemma \ref{hamiltonian lemma} hold for $G_r(t) = (X_{JQ}(r, t))^*X_{JQ}(r, t)$ on the segment $I=[r, r + 4l_0]$. Indeed, the properties of the ordered exponent instantly give that $G_r(t)\ge 0$ and $\det G_r(t) = 1$ for all $t\in I$. Furthermore, if $v\in \R^2$ is a unit vector, then
\begin{gather*}
    \langle G_r(t)v,v \rangle = \norm[X_{JQ}(r, t)v]^2_2 \approx 1,
\end{gather*}
because of inequality \eqref{norm inequality for ordered exponent} and the assertion $Q\in L^2$. An application of Lemma \ref{hamiltonian lemma} gives 
\begin{align*}
    \frac{1}{4}E_Q^{(2l_0)}(r) &= \det \langle G_r \rangle_I - 1\le \max_{n\in\{1,2,3\}}\left[\det\langle G_r\rangle_{I_n} - 1\right]
    \\
    &= \frac{1}{4}\max\left[E_Q^{(l_0)}(r), E_Q^{(l_0)}(r + l_0), E_Q^{(l_0)}(r + 2l_0)\right] \ls f(r).
\end{align*}
Therefore \eqref{decay of the scaled entropy in the lemma} holds for $l = 2 l_0$. Iterating this procedure we get that it also holds for $l = 2^n l_0$, for every integer $n > 0$. It remains for us to notice that \eqref{decay of the scaled entropy in the lemma} for $l = l_0$ implies \eqref{decay of the scaled entropy in the lemma} for all $l < l_0$ (see formula 5.1 in \cite{Bessonov2022}).
\end{proof}
\subsection{Proof of Theorem \ref{theorem on equivalence}}
\begin{proof}[Proof of Theorem \ref{theorem on equivalence}]
As we mentioned at the beginning of Section \ref{proof of Theorem 2}, the assertion $p,q\in L^2(\R_+)$ is equivalent to $a\in L^2(\R_+)$, where $a$ is the coefficient of the Krein system corresponding to $D_Q$. Therefore Corollaries \ref{corollary from the entropy decay to the H2 } and \ref{corollary from extension to the decay of entropy} give us the inequalities
\begin{gather*}
    \min(4, 4 \Delta_{\Pi})\le \Delta_{E}\le 8\Delta_{\Pi}.
\end{gather*}
If $\Delta_{\Pi} < 1$ then this is exactly the claim of Theorem \ref{theorem on equivalence} otherwise let us use the following rescaling argument. Let $v > 0$ and consider Krein system \eqref{Krein system} with the coefficient $a^{(v)}\colon t\mapsto va(vr)$. Let
\begin{gather*}
    P_*^{(v)},\quad \Pi^{(v)},\quad H_Q^{(v)},\quad E_{Q^{(v)}},\quad \Delta^{(v)}_{\Pi}\quad\mbox{ and } \quad\Delta^{(v)}_{E}
\end{gather*}
 be the solution, the inverse \Szego function, the Hamiltonian, the entropy function and the parameters from Theorem \ref{theorem on equivalence} corresponding to this scaled Krein system. Simple calculation show that $P_*^{(v)}(r, \lambda) = P_*^{(v)}(vr, \lambda/r)$ and therefore
 \begin{gather*}
     \Pi^{(v)}(\lambda) = \Pi(\lambda / v), \qquad\Delta_{\Pi}^{(v)} = v\Delta_{\Pi}.
 \end{gather*}
If $v$ is such that $v\Delta_{\Pi} < 1$, then the conclusions of the theorem hold for the scaled Krein system. From the definition of the Hamiltonian we have $H_Q^{(v)}(t) = H_Q(vt)$ and hence $E_{Q^{(v)}}(r) = E^{(v)}_Q(vr)$, where $E^{(v)}_Q$ is the scaled entropy, see \eqref{definition of the scaled entropy}. From Theorem \ref{theorem decay of different entropies} it follows that $\Delta_{E}^{(v)} = v{\Delta}_{E}$. Therefore we have 
\begin{gather*}
   \frac{\Delta_{\Pi}}{\Delta_{E}} =\frac{\Delta^{(v)}_{\Pi}}{\Delta_{E}^{(v)}} \in [4,8].
\end{gather*}
The proof of the theorem is concluded.

\end{proof}
\section{Some applications}\label{sufficient conditions}
In this section we prove Theorems \ref{theorem on resonances} and \ref{theorem on sufficient conditions} and discuss the sharpness of the inequalities in Theorems \ref{Entropy for extension} and \ref{theorem on equivalence}.
\subsection{Proof of Theorem \ref{theorem on sufficient conditions}}
\subsubsection{Off-diagonal potentials.}\label{Section Sufficient condition in the diagonal case}
Recall that we have the potential $Q = \left(\begin{smallmatrix}0&p\\p&0\end{smallmatrix}\right)$ where $p$ is a real-valued function $p$ such that $\sup_{t\ge r}\left|\int_{r}^{t}p(s)\,ds\right| = O\left(e^{-\delta r}\right)$ as $r\to\infty$. In this case both $N_Q$ and  $H_{Q}$ defined by \eqref{hamiltonian for dirac} are diagonal and can be calculated explicitly. Namely,
\begin{gather*}
    N_{Q}(t) = 
    \begin{pmatrix}
        \exp\left(-g_0(t)\right) &0
        \\
        0 & \exp\left(g_0(t)\right)
    \end{pmatrix},
    \qquad 
    H_{Q}(t) = 
    \begin{pmatrix}
        \exp\left(-2g_0(t)\right) &0
        \\
        0 & \exp\left(2g_0(t)\right)
    \end{pmatrix},
\end{gather*}
where $g_r(t) = \int_r^t p(s)\, ds$. Moreover,
\begin{align}
    \nonumber
    E_{Q}(r) = \det\int_{r}^{r + 2}  H_{Q}(t)\,dt - 4 &= \int_r^{r + 2} e^{-2g_0(t)}\,dt \cdot \int_r^{r + 2} e^{2g_0(t)}\,dt - 4
    \\
    \label{entropy expression via the g_r function}
    &= \int_r^{r + 2} e^{-2g_r(t)}\,dt \cdot \int_r^{r + 2} e^{2g_r(t)}\,dt - 4.
\end{align}
The assumption on $p$ of Theorem \ref{theorem on sufficient conditions} can be rewritten in the form $ \sup_{t\ge r} |g_r(t)| \ls e^{-\delta r/2}$. Hence the inequalities $\displaystyle \sup_{t\ge r}\left|e^{2g_r(t)} -2g_r(t) - 1\right| \ls e^{-\delta r}$ and $\displaystyle \sup_{t\ge r}\left|e^{-2g_r(t)} + 2g_r(t) - 1\right| \ls e^{-\delta r}$
hold and \eqref{entropy expression via the g_r function} gives $E_{Q}(r)\ls e^{-\delta r}$. This completes the proof in the off-diagonal case.
\subsubsection{General situation.}\label{sufficient condition general situation}
Recall \eqref{another representation of scaled entropy} that the the entropy function can be calculated as 
\begin{gather}\label{tmp def entropy}
    E_Q(r) = \det\int_{r}^{r + 2}  (X_{JQ}(r,t))^*X_{JQ}(r,t)\,dt - 4,
\end{gather}
where $X_{JQ}$ is the ordered exponent of $JQ$. We claim that under the assumptions of Theorem \ref{theorem on sufficient conditions} $X_{JQ}$ is close to the identity operator. Indeed, the operator $\int_r^tQ(s)\,ds$ is symmetric and has zero trace, hence  $\norm[\int_r^tQ(s)\,ds] = \sqrt{\left(\int_r^t p(s)\,ds\right)^2 + \left(\int_r^t q(s)\,ds\right)^2}\ls e^{-\delta r}$. Apply this norm inequality to explicit formula \eqref{explicit representaion ordered exponent} for the ordered exponent. We have
\begin{gather*}
    \norm[X_{JQ}(r,t) - I] = \norm[\sum_{m\ge 1}\int_r^t JQ(t_1) \int_r^{t_1} JQ(t_2)\int_r^{t_2}\ldots \int_r^{t_{m - 1}} JQ(t_m)\, dt_m\,\ldots dt_3\,dt_2\, dt_1]
    \\
    \le \sum_{m\ge 1}\int_r^t\int_r^{t_1}\ldots\int_r^{t_{m-2}} \prod_{k = 1}^{m-1} \norm[Q(t_k)]\cdot\norm[\int_r^{t_{m - 1}} Q(t_m)\, dt_m]  dt_{m - 1}\ldots dt_2\, dt_1
    \\
    \le \sup_{s\ge r}\norm[\int_r^{s} Q(s)\, ds]\cdot \sum_{m\ge 1}\int_r^t\int_r^{t_1}\ldots\int_r^{t_{m-2}} \prod_{k = 1}^{m-1} \norm[Q(t_k)] dt_{m - 1}\ldots dt_2\, dt_1
    \\
    =\sup_{s\ge r}\norm[\int_r^{s} Q(s)\, ds]\cdot
    \sum_{m\ge 1}\frac{\left(\int_r^t\norm[Q(t_1)]\,dt_1\right)^{m - 1}}{(m-1)!} 
    \\
    \ls e^{-\delta r}\cdot \exp\left[\int_r^t\norm[JQ(t_1)]\,dt_1\right] \le e^{-\delta r}\cdot \exp\left[\int_r^t|p(t_1)| + |q(t_1)|\,dt_1\right]. 
\end{gather*}
This for $t\in [r, r + 2]$ implies $\norm[X_{JQ}(r,t) - I]\ls e^{-\delta r}$ hence $(X_{JQ}(r,t))^*X_{JQ}(r,t) = I + M_r(t)$, where $M_r$ is a matrix-valued function such that $\norm[M_r(t)] \ls e^{-\delta r}$. To conclude the proof, substitute the latter into \eqref{tmp def entropy}:
\begin{align*}
    E_Q(r) = \det\int_{r}^{r + 2}  \left(I + M_r(t)\right)\,dt - 4 =
 \det\left(2I + \int_{r}^{r + 2} M_r(t)\right)\,dt - 4 \ls e^{-\delta r}.
\end{align*}
\subsection{On the sharpness of results}\label{section: on the optimality of the constants}
It Theorems \ref{theorem on equivalence} and \ref{Entropy for extension}  we established  inequalities between the rate of exponential decay of the entropy function $E_Q$ and the width of the strip, where the inverse \Szego function and the Weyl function have proper extensions. In this section we give an example of a potentials for which some of the estimates in these two theorems are precise.  Consider Krein system \eqref{Krein system} with the coefficient $a(r) =e^{-r}$. Then the potential of the corresponding Dirac operator $D_{Q}$ is given by $Q = \left(\begin{smallmatrix}0 & 2e^{-2r}\\2e^{-2r} & 0\end{smallmatrix}\right)$.
Let us show that for this $Q$ the corresponding inverse \Szego function does not extend into $\Omega_{\delta}$ for any $\delta > {\Delta}_E/4$ and, moreover, the equality $4 {\Delta}_{\Pi} = {\Delta}_{E}$ holds.
\medskip

The potential Q is off-diagonal hence the entropy can be calculated by formula \eqref{entropy expression via the g_r function}. In our situation we have $g_r(t) = \int_r^t 2e^{-2t_1}\, dt_1 = e^{-2r}(1 - e^{2(t - r)})$. 
Then \eqref{entropy expression via the g_r function} gives $E_Q(r)\approx e^{-4r}$, $\Delta_E = 4$ and, by Theorem \ref{Entropy for extension}, $\Pi$ extends analytically into $\Omega_1$. Let us show that $\Pi$ does not extend into $\Omega_{\delta}$ for any $\delta > 1$.
The inequality $|P(r,\lambda)| \le |P_*(r,\lambda)|$ in $\ol{\Cm_+}$ and the Gronwall-Bellman inequality for differential equation \eqref{Krein system} for $P_*$  give $|P_*(r,\lambda)|\approx 1$ uniformly in $\lambda\in\ol{\Cm_+}$ and $r\ge 0$. Furthermore, we have
\begin{gather*}
    \left|\frac{\d}{\d r}P_*(r,\ol{\lambda})\right| = \left|a(r)P(r,\ol{\lambda})\right| \stackrel{\eqref{Krein system reflection formula}}{=}\left|a(r) e^{i\ol{\lambda}r}\ol{P_*(r,\lambda)}\right| = e^{-(1 - \Im\lambda)r}\left|P_*(r,\lambda)\right|.
\end{gather*}
Therefore $\norm[\frac{\d}{\d r}P_*(r,\ol{\lambda})]_{L^1(\R_+)}$ is uniformly bounded in $\{\lambda\colon 1 - \eps \ge \Im\lambda \ge 0\}$ for every $\eps > 0$. Consequently $\Pi$ is uniformly bounded in $\Omega_{\delta}$ for every $\delta < 1$ and ${\Delta}_{\Pi} \ge 1 = {\Delta}_{E}/4$. Theorem \ref{theorem on equivalence} then implies $4 {\Delta}_{\Pi} = {\Delta}_{E}$.
\medskip

On the other hand, for $\lambda = ih$, the Krein system has only real-valued coefficients hence $P_*(r, ih)$ is real-valued; $P_*(r, ih)$ is positive for every $r\ge 0$ because  $P_*(r, ih)\neq 0$ for every $r\ge 0$. Therefore the relation 
\begin{gather*}
    \frac{\d}{\d r}P_*(r,-ih) = e^{-(1 - h)r}P_*(r,ih)\approx e^{-(1 - h)r}
\end{gather*}
holds uniformly for $h\ge 0$. It follows that
\begin{gather*}
    |\Pi(-ih)| = \lim_{r\to\infty} |P_*(r,- ih)|\approx 1 + \int_0^{\infty}e^{-(1 - h)\rho}d\rho = \frac{h}{1 - h}.
\end{gather*}
Therefore $|\Pi(-ih)|\to\infty$ as $h\to 1^-$ and $\Pi$ cannot be analytically extended to the point $\lambda = -i$ and especially into any horizontal half-plane containing $-i$. 
\subsection{Resonances. Proof of Theorem \ref{theorem on resonances}}\label{section: resonances}
Recall definition \eqref{definition of class E} of the class $\mathcal{E}$ of potentials with super-exponentially decaying entropy. By Theorem \ref{Entropy for extension}, for any $Q\in\mathcal{E}$ the corresponding inverse \Szego function $\Pi = \Pi_Q$ and the Weyl function $m=m_Q$ extend into the whole complex plane. We denote these extensions with the same symbols $\Pi$ and $m$. By the definition, the scattering resonances of $D_Q$ are the poles of $m$. Let us show that they are precisely the zeroes of $\Pi$. Indeed, $Q\in \mathcal{E}$ iff $-Q\in \mathcal{E}$ hence relation \eqref{Weyl function via the Szego function} implies that if $z$ is a pole of $m_Q$ of multiplicity $n$, then $z$ is a zero of $\Pi$ of multiplicity not greater than $n$. To establish that the multiplicities are the same we need to show that $\Pi$ and $\hat \Pi$ have no common zeroes, which in its turn follows from
\begin{gather}\label{relation between two szego functions}
    \ol{\Pi(\ol{z})}\hat{\Pi}(z) + \Pi(z)\ol{\hat{\Pi}(\ol{z})} = 2, \qquad z\in \Cm.
\end{gather}
When $Q$ has compact support the latter relation simply follows from (75) in \cite{Denisov2006}; in our situation we need to consider the dual regularized Krein system more carefully.
The coefficients of the dual regularized Krein system are not much different from the original system of \eqref{diff equation for tilde P_*} and \eqref{diff equation for tilde P}. Namely (see the remark after Lemma 7.3 in \cite{Bessonov2021}), we need to change the sign of $f_1$ and leave $f_2$ the same. In a matrix form two regularized systems become
\begin{gather*}
    \frac{\d}{\d r}
    \begin{pmatrix}
    \tilde{P}_r^*
    \\
    \tilde{P}_r
    \end{pmatrix}
    =
    \begin{pmatrix}
    izf_2 & (z -i)f_1
    \\
    (z + i)\ol{f_1} & iz(1 - f_2)
    \end{pmatrix}
    \begin{pmatrix}
    \tilde{P}_{r}^*
    \\
    \tilde{P}_r
    \end{pmatrix},
    \\
    \frac{\d}{\d r}
    \begin{pmatrix}
    \tilde{P}_{r,d}^*
    \\
    \tilde{P}_{r,d}
    \end{pmatrix}
    =
    \begin{pmatrix}
    izf_2 & -(z -i)f_1
    \\
    -(z + i)\ol{f_1} & iz(1 - f_2)
    \end{pmatrix}
    \begin{pmatrix}
    \tilde{P}_{r,d}^*
    \\
    \tilde{P}_{r,d}
    \end{pmatrix},
\end{gather*}
where the functions with the subindex $d$ form the regularized dual system. Denote the matrix in the first equality by $V$. A simple transformation gives
\begin{gather*}
    \frac{\d}{\d r}
    \begin{pmatrix}
    \tilde{P}_r^*
    \\
    \tilde{P}_r
    \end{pmatrix}
    =
    V
    \begin{pmatrix}
    \tilde{P}_r^*
    \\
    \tilde{P}_r
    \end{pmatrix},
    \quad 
    \frac{\d}{\d r}
    \begin{pmatrix}
    \tilde{P}_{r,d}^*
    \\
    -\tilde{P}_{r,d}
    \end{pmatrix}
    =
    V
    \begin{pmatrix}
    \tilde{P}_{r,d}^*
    \\
    -\tilde{P}_{r,d}
    \end{pmatrix}.
\end{gather*}
Hence the solution $X$ of the $X'= VX$  with $X(0,z) = I$ is given by
\begin{gather*}
    X = \frac{1}{2}
    \begin{pmatrix}
    \tilde{P}_{r}^* & \tilde{P}_{r,d}^*
    \\
    \tilde{P}_{r} & -\tilde{P}_{r,d}
    \end{pmatrix}
    \begin{pmatrix}
    1 & 1
    \\
    1 & -1
    \end{pmatrix}.
\end{gather*}
We have
\begin{align*}
    \frac{1}{2}\left(\tilde{P}_{r}^*(z)\tilde{P}_{r,d}(z) + \tilde{P}_{r,d}^*(z)\tilde{P}_{r}(z) \right) = \det X(r,z) &= \exp\left(\int_0^r\trace V(t,z)\,dt\right) 
    \\
    &= \exp\left(\int_0^r iz\,dt\right)= e^{irz}.
\end{align*}
By \eqref{inversion formula} the latter is equivalent to
\begin{gather*}
    \tilde{P}_{r}^*(z)\ol{\tilde{P}_{r,d}^*(\ol{z})} + \tilde{P}_{r,d}^*(z)\ol{\tilde{P}_{r}^*(\ol{z})} = 2,\qquad z\in \Cm.
\end{gather*}
If $Q\in \mathcal{E} $ then from the proof of Theorem \ref{Entropy for extension} it follows that $\displaystyle\Pi(z) = \lim_{r\to\infty} \tilde{P}_r^*(z) $ and $\displaystyle\hat{\Pi}(z) = \lim_{r\to\infty} \tilde{P}_{r,d}^*(z) $ for every $z\in\Cm$. Hence \eqref{relation between two szego functions} follows and $\Pi(z) = \hat{\Pi}(z) = 0$ is not possible.
\medskip

Thus the resonances are the zeroes of the entire function $\Pi$. The counting function of the zeroes is tightly connected to the order of $\Pi$.
\begin{lemma}\label{order of the szego function for resonances}
If $Q\in \mathcal{E}_{\alpha}$ for some $\alpha > 1$ then $\Pi$ is of order not greater than $\displaystyle\frac{\alpha}{\alpha - 1}$. 
\end{lemma}
\begin{proof}
We need to show that for $z\in \Cm$, the inequality 
\begin{gather*}
    |\Pi(z)|\ls \exp\left(C|z|^{\frac{\alpha}{\alpha - 1}}\right)
\end{gather*}
holds for some constant $C$. From Lemma \ref{entropy and absolute continuity} we know the stronger bound $\Pi(z)\ls \exp\left(C|z|\right)$ in $\Cm_+$. The limit relation $\displaystyle\Pi(z) = \lim_{r\to\infty} \tilde{P}_r^*(z) $ holds for every $z\in\Cm$ hence \eqref{trivial inequality for the differential equation} gives
\begin{gather*}
    |\Pi(z)| \ls (|z| + 1)e^{c(|z| + 1)} \int_0^{\infty}\left(\sqrt{|\K'(\rho/2)|} + |\K'(\rho/2)|\right)e^{\rho|\Im z|}\,d\rho,\quad \Im z \le 0.
\end{gather*}
We have $E_Q(r)\ls \exp\left(-r^{\alpha}\right)$ hence $\K(r)\ls \exp\left(-r^{\alpha}\right)$. It follows that
\begin{align*}
    |\Pi(z)|&\ls (|z| + 1)e^{c(|z| + 1)}\int_0^{\infty}\exp\left(\rho|\Im z| - \rho^{\alpha}/2^{\alpha + 1}\right)\,d\rho
    \\
    &\ls (|z| + 1)e^{c(|z| + 1)}\exp\left(c_1|\Im z|^{\frac{\alpha}{\alpha - 1}}\right) \ls \exp\left(c_2|z|^{\frac{\alpha}{\alpha - 1}}\right).
\end{align*}
\end{proof}
\begin{proof}[Proof of Theorem \ref{theorem on resonances}]
Let $Q_1$ and $Q_2$ be two potentials from $\mathcal{E}$ with the same set of resonances with multiplicities $\{z_k\}_{k = 0}^{\infty}$. Let $\sigma_i, \Pi_i, m_i$ and $H_i$, for $i = 1,2$ be the spectral measures, the inverse \Szego functions, the Weyl functions and the Hamiltonians corresponding to $Q_1$ and $Q_2$. Denote by $E_n$ the elementary  Weierstrass factor,
\begin{align*}
    E_n(z) =
    \begin{cases} 
    (1-z), & n=0, 
    \\
    (1-z)\exp \left( \frac{z}{1}+\frac{z^2}{2}+\cdots+\frac{z^n}{n} \right), & n > 0.
    \end{cases}
\end{align*}
By Lemma \ref{order of the szego function for resonances}, both $\Pi_1$ and $\Pi_2$ are of finite exponential order; let $n$ be an integer such that  the order of $\Pi_1$ and $\Pi_2$ does not exceed $n$. Hadamard factorization theorem, see Section 4.2 in \cite{Levin}, gives that 
\begin{gather*}
    \Pi_1(z) = e^{g_1(z)} \prod_{k = 0}^{\infty}E_n\left(\frac{z}{z_k}\right)
    \;\text{ and } \;
    \Pi_2(z) = e^{g_2(z)} \prod_{k = 0}^{\infty}E_n\left(\frac{z}{z_k}\right),
\end{gather*}
where $g_1$ and $g_2$ are polynomials of degree not greater than $n$.
Hence $\Pi_1(z)\Pi_2^{-1}(z) = e^{g(z)}$ with $g = g_1 - g_2$.
By Lemma \ref{entropy and absolute continuity} there exists $c > 0$ such that $|\Pi_1(z)|\ls e^{c|z|}$ and $e^{-c|z|}\ls  |\Pi_2(z)|$ for $z\in \ol{\Cm_+}$.
Consequently, $\displaystyle\left|e^{g(z)}\right|\ls e^{2c|z|}$ or $\Re g(z) \ls |z| + 1$ for $z\in \ol{\Cm_+}$, which is possible only if $\deg g \le 1$. Let $g(z) = a_1z + a_2$, where $a_1$ and $a_2$ are some complex constants. Because of the \Szego class assumption \eqref{Szego class definition} we have
\begin{gather*}
    \frac{\Re g}{x^2 + 1} = \frac{\log|\Pi_1| - \log|\Pi_2|}{x^2 + 1}\in L^1(\R). 
\end{gather*}
It follows that $\Re a_1 = 0$ and $|\Pi_1(x)| = |\Pi_2(x)| e^{\Re a_2}$, $x\in \R$. Next, by Lemma \ref{entropy and absolute continuity}, both $\sigma_1$ and $\sigma_2$ are a.\,c.\ hence $d\sigma_1 = e^{-2\Re a_2}d\sigma_2$; relation \eqref{spectral measure and weyl function} then gives $m_1 = am_2 + b$ for some constants $a > 0$ and $b$. Proof of Lemma 2.4 in \cite{Bessonov2021} implies $H_2(t) = A^* H_1(t) A$ with $A = \left(\begin{smallmatrix}1/\sqrt{a} &b/\sqrt{a}\\0 &\sqrt{a}\end{smallmatrix}\right)$.
For $t = 0$ it becomes $ I = A^*A = \left(\begin{smallmatrix} 1/a &b/a \\   b/a &b^2/a + a \end{smallmatrix}\right)$. Hence $a = 1$, $b = 0$ and the Weyl functions $m_1$ and $m_2$ coincide. Therefore $Q_1= Q_2$ and the proof is finished.
\end{proof}

\bibliographystyle{plain}
\bibliography{references}

\end{document}